\documentclass{amsart}
\usepackage[OT2,T1]{fontenc}
\usepackage[russian,USenglish]{babel}

\usepackage{amssymb}
\usepackage{bbm}
\usepackage{mathrsfs}
\usepackage{xypic}

\makeatletter
\def\hsmash{\relax 
  \ifmmode\def\next{\mathpalette\mathhsm@sh}\else\let\next\makehsm@sh
  \fi\next}
\def\makehsm@sh#1{\setbox\z@\hbox{#1}\finhsm@sh}
\def\mathhsm@sh#1#2{\setbox\z@\hbox{$\m@th#1{#2}$}\finhsm@sh}
\def\finhsm@sh{\wd\z@\z@ \box\z@}
\makeatother

\newtheorem{fac}{Fact}[section]

\newtheorem{lem}[fac]{Lemma}
\newtheorem{prop}[fac]{Proposition}
\newtheorem{theo}[fac]{Theorem}

\theoremstyle{definition}

\newtheorem{defi}[fac]{Definition}

\theoremstyle{remark}
\newtheorem{rem}[fac]{Remark}

\def\cyr#1{\foreignlanguage{russian}{#1}}

\newcommand{\Br}{\mathop{\text{\rm Br}}\nolimits}

\newcommand{\ev}{\mathop{\text{\rm ev}}\nolimits}

\newcommand{\bbP}{{\mathbbm P}}
\newcommand{\bbQ}{{\mathbbm Q}}
\newcommand{\bbR}{{\mathbbm R}}
\newcommand{\bbZ}{{\mathbbm Z}}
\newcommand{\A}{{\mathbbm A}}

\newcommand{\F}{{\mathbbm F}}
\newcommand{\N}{{\mathbbm N}}
\newcommand{\Q}{{\mathbbm Q}}
\newcommand{\R}{{\mathbbm R}}
\newcommand{\Z}{{\mathbbm Z}}

\newcommand{\calH}{{\mathscr{H}}}
\newcommand{\calJ}{{\mathscr{J}}}
\newcommand{\calK}{{\mathscr{K}}}

\newcommand{\calQ}{{\mathscr{Q}}}
\newcommand{\calR}{{\mathscr{R}}}
\newcommand{\calS}{{\mathscr{S}}}
\newcommand{\calT}{{\mathscr{T}}}

\newcommand{\bfd}{\mathbf{d}}
\newcommand{\bfk}{\mathbf{k}}
\newcommand{\bfl}{\mathbf{l}}
\newcommand{\bfq}{\mathbf{q}}
\newcommand{\bft}{\mathbf{t}}
\newcommand{\bfQ}{\underline{Q}}

\newcommand{\gam}{{\gamma}}
\newcommand{\del}{{\delta}}
\newcommand{\eps}{{\varepsilon}}
\newcommand{\lam}{{\lambda}}
\newcommand{\sig}{{\sigma}}

\def\alp{{\alpha}} 
\def\bet{{\beta}}  
\def\gam{{\gamma}} 

\def\Gam{{\Gamma}}
\def\del{{\delta}} 

\def\zet{{\zeta}}  

\def\tet{{\theta}}  

\def\lam{{\lambda}}  
\def\lamtil{{\widetilde{\lam}}}

\def\Sig{{\Sigma}}

\newcommand{\loc}{{{\text{\rm loc}}}}

\newcommand{\vol}{{{\text{\rm vol}}}}

\newcommand{\RBR}{{{R_{\text{Br}}}}}
\newcommand{\ra}{{{r_1}}}

\newcommand{\mmod}[1]{\,\,\text{mod}\,\,#1}

\newcounter{abc}

\newcounter{iii}

\def\rightend#1#2{{%
 \leavevmode\nobreak\hskip .5em plus 1fil
 \penalty600 \hskip 0pt plus -1filll
 \vadjust{}\nobreak\hskip 0pt plus 1filll%
 #1\parfillskip=#2\relax \par}}

\def\eop{\ifmmode\rule[-22pt]{0pt}{1pt}\ifinner\tag*{$\square$}\else\eqno{\square}\fi\else\rightend{$\square$}{0pt}\fi}

\title[Del Pezzo surfaces of degree four violating the Hasse principle]{On the number of certain Del Pezzo surfaces \\of degree~four violating the Hasse principle}

\begin{document}

\author{J\"org Jahnel}

\address{D\'epartement \!Mathematik\\ Universit\"at \!Siegen\\ Walter-Flex-Stra\ss e~3\\ D-57068 Siegen\\ Germany}
\email{jahnel@mathematik.uni-siegen.de}
\urladdr{http://www.uni-math.gwdg.de/jahnel}

\author{Damaris Schindler}

\address{Hausdorff Center for Mathematics\\ Endenicher Allee 62\\ D-53115 Bonn\\
Germany}
\email{damaris.schindler@hausdorff-center.uni-bonn.de}
\urladdr{http://www.math.uni-bonn.de/people/dschindl}


\date{\today}



\begin{abstract}
We give an asymptotic expansion for the density of del Pezzo surfaces of degree four in a certain Birch Swinnerton-Dyer family violating the Hasse principle due to a Brauer-Manin obstruction. Under the assumption of Schinzel's hypothesis and the finiteness of Tate-Shafarevich groups for elliptic curves, we obtain an asymptotic formula for the number of all del Pezzo surfaces in the family, which violate the Hasse principle.
\end{abstract}

\maketitle

\section{Introduction}
The goal of this paper is to establish an asymptotic formula for the density of del Pezzo surfaces of degree four in a certain family of Birch Swinnerton-Dyer type, which violate the Hasse principle due to a Brauer-Manin obstruction. More precisely, let $D\in \bbZ$ be some fixed discriminant, which is not a perfect square, and $A,B\in \bbZ$. Let $S^{(D;A,B)}$ be the surface in $\bbP^4$ given by the system of quadrics
\begin{equation}\label{eqn0}
\begin{split}
t_0t_1&=t_2^2-Dt_3^2,\\
(t_0+At_1)(t_0+Bt_1)&=t_2^2-Dt_4^2.
\end{split}
\end{equation}
If $A$ and $B$ are chosen in a way such that $A\neq B$, $AB\neq 0$ and $A^2-2AB+B^2-2A-2B+1\neq 0$, then $S^{(D;A,B)}$ is a smooth del Pezzo surface of degree four. We are interested in the frequency how often the surface $S^{(D;A,B)}$ fails the Hasse principle. In order to formulate a reasonable counting question, we need to introduce some height function, according to which we order the del Pezzo surfaces in the family above. For a fixed discriminant $D$, we use a naive height given by
\begin{equation*}
H(S^{(D;A,B)}):=\max\{|A|,|B|\}.
\end{equation*}
Let $R_D(N)$ be the number of integers $|A|,|B|\leq N$ such that $S^{(D;A,B)}$ is smooth and fails the Hasse principle. We can hence characterize a tuple $(A,B)$ with $|A|,|B|\leq N$, which is counted by $R_D(N)$, by the property that $S^{(D;A,B)}$ is smooth, $S^{(D;A,B)}(\bbQ_p)\neq \emptyset $ for all primes $p$ including the infinite prime, and such that $S^{(D;A,B)}(\bbQ)=\emptyset$. It is conjectured that all failures of the Hasse principle for del Pezzo surfaces in general can be explained by some Brauer-Manin obstruction. Hence we introduce the counting function $\RBR (N)$ to be the number of all surfaces $S^{(D;A,B)}$ in the family (\ref{eqn0}) of height at most $N$ with the property that there is a Brauer-Manin obstruction to the Hasse principle for $S^{(D;A,B)}$. In particular, we have the lower bound $R_D(N) \geq \RBR (N)$. Our first main theorem gives an asymptotic expansion for $\RBR (N)$. 

\begin{theo}\label{theo1}
Let $D>1$ be some positive squarefree integer, which satisfies $D\equiv 1$ modulo $8$. For any $P\geq 0$, there are real constants $C_k$ such that
\begin{equation*}
\RBR (N)=\frac{4N^2}{(\log 2N)^{1/4}} \sum_{k=0}^{2P} \frac{C_k}{(\log 2N)^{k/2}} +O_{D,P} \left(\frac{N^2}{(\log N)^{3/4+P}}\right).
\end{equation*}
The constants $C_k$ have explicit descriptions as in equation (\ref{defCk}) and (\ref{defC0}).
Moreover, the leading constant $C_0$ is positive.
\end{theo}

We note that the implied constant in the error term depends ineffectively on $P$ due to possible Siegel zeros of Dirichlet $L$-functions.\\

Moreover, we note that we can always reduce to the case where $D$ is squarefree by substituting $t_3=d^{-1} t_3'$ and $t_4=d^{-1} t_4'$ if $d^2|D$ for some positive integer $d$.\\

In \cite{VAV}, V\'arilly-Alvarado and Viray have shown that the Brauer-Manin obstruction to the Hasse principle (and weak approximation) is the only one for the family (\ref{eqn0}) under the assumption of Schinzel's hypothesis and the finiteness of Tate-Shafarevich groups of elliptic curves. Under these two conjectures, we hence conclude that we also obtain an asymptotic expansion for $ R_D(N)$. 

\begin{theo}\label{theo2}
Let $D$ be as in Theorem \ref{theo1} and $P\geq 0$. Assume Schinzel's hypothesis and the finiteness of Tate-Shafarevich groups of elliptic curves. Then 
\begin{equation*}
R_D(N)= \frac{4N^2}{(\log 2N)^{1/4}} \sum_{k=0}^{2P} \frac{C_k}{(\log 2N)^{k/2}} +O_{D,P} \left(\frac{N^2}{(\log N)^{3/4+P}}\right),
\end{equation*}
with real constants $C_k$ given as in Theorem \ref{theo1}.
\end{theo}

\begin{rem}
Theorem \ref{theo1} can be used to conclude that the set of del Pezzo surfaces of degree four, that are counterexamples to the Hasse principle, is Zariski dense in the moduli scheme. Indeed, one can argue as in Theorem 6.11 in \cite{JS} and note that the density estimate for $\RBR (N)$ implies that the set of $(A,B)$, for which $S^{(D;A,B)}$ is a counterexample to the Hasse principle, cannot be contained in a finite union of curves in $\A^2$.
\end{rem}

Next we compare our result from Theorem \ref{theo1} with the number of del Pezzo surfaces in the family that are everywhere locally soluble. We let $R_D^{\loc}(N)$ be the number of $|A|,|B|\leq N$ such that $S^{(D;A,B)}(\bbQ_p)\neq \emptyset$ for all primes $p$, including the infinite prime. 

\begin{prop}\label{prop1}
Assume that $D$ is some positive, squarefree integer with $D\equiv 1$ modulo $8$. Then there exists a positive constant $c_\loc$ such that one has
\begin{equation*}
R_D^{\loc}(N)=c_\loc N^2+O(N^{2-\tet(D)}),
\end{equation*}
for some $\tet(D) >0$. 
The constant $c_\loc$ has an explicit description in equation (\ref{cloc}) in section \ref{local}. In particular, it is a product of local densities.
\end{prop}

In section \ref{local}, we give an elementary proof of Proposition \ref{prop1}. Alternatively, one should be able to use the methods from the papers of Poonen and Stoll in \cite{PS99a} and \cite{PS99b} as used in Theorem 3.6 in work of Poonen and Voloch \cite{PV} or work of Ekedahl \cite{Eke}. However, it turns out that except for a finite number, all the local densities in our problem are identically equal to $1$ and hence we can pursue an easier proof. Moreover, we obtain an explicit error term with a power saving of the main parameter.\par

If $D$ is chosen suitably as in the assumptions of our main theorems, then Proposition~\ref{prop1} shows that a positive proportion of surfaces in the family~(\ref{eqn0}) is locally soluble. The family~(\ref{eqn0}) is built in a way that in the generic case one obtains a non-trivial Brauer-group isomorphic to~$\Z/2\Z$. However, in most cases one would only expect that weak approximation is obstructed, but the Hasse principle still holds. Theorem \ref{theo1} and Theorem \ref{theo2} verify this expectation for the Hasse principle in a quantitative way.\par

Similar questions for other families of algebraic varieties have been studied before. In \cite{Bha}, Bhargava considers families of genus one curves and shows among other results that a positive proportion of plane cubics fail the Hasse principle. In a similar spirit, Browning and Newton \cite{BroNew} study twists of norm one tori and find that a positive proportion of rational numbers fail the Hasse norm principle in the case of a non-trivial knot group. The situation for the degree four del Pezzo surfaces in our family is different in the sense that only on a thin subset one observes failure of the Hasse principle due to a Brauer-Manin obstruction. This phenomenon is closer to the observations of La Br\'eteche and Browning \cite{BB2} on the failure of the Hasse principle for a certain family of Ch\^{a}telet surfaces. Similarly to our situation, they find a positive proportion of locally soluble surfaces and only a thin set failing the Hasse principle, with a density decaying like $\smash{\sim \frac{1}{(\log N)^{1/4}}}$ as in our Theorems \ref{theo1} and \ref{theo2}. The study of a certain family of coflasque tori in \cite{BB1} shows a similar behaviour.\par


In order to count counterexamples to the Hasse principle in the family (\ref{eqn0}), we need to understand the Brauer group of the variety and its evaluation on the local points $S^{(D;A,B)}(\Q_\nu)$ for any place~$\nu$ sufficiently well. Our analysis in this direction, in particular criteria for the constancy of the evaluation of a Brauer group element on $S^{(D;A,B)}(\Q_\nu)$, builds on and generalizes part of our earlier work in \cite{JS}. For inert primes, we have a rather precise criterion (see Lemma \ref{inert}), whereas for ramified primes the situation remains to some extent unsolved. We circumvent the problem in using the continuity properties of Brauer classes, see Lemma \ref{lem1}.\par
Note that in our setting it is enough to consider algebraic Brauer classes. Since del Pezzo surfaces are rational varieties, their Brauer group is trivial after passing to some algebraic closure, see Remark 1.3.8 in \cite{Co} and III, Example 8.7.ii) in \cite{Ja} as well as Theorem 42.8 in \cite{Ma}. Hence, in the usual notation we have $\Br X= \Br_1 X= \ker [\Br X \rightarrow \Br \overline{X}]$.\par

The structure of this paper is as follows. In section \ref{local}, we study the number of locally soluble del Pezzo surfaces in our family (\ref{eqn0}) and prove Proposition \ref{prop1}. In section \ref{section3}, we study the action of the Brauer group at inert primes and give explicit criteria for its evaluation on $S^{(D;A,B)}(\Q_p)$. We use these criteria in section \ref{section4} to give asymptotics for counting functions related to $\RBR (N)$. First, we additionally fix $A$ and $B$ in congruence classes modulo some integer $T$ that is composed of primes dividing the discriminant $D$. We use these asymptotics in the final section to prove the main theorem \ref{theo1}.\par

We note that all implicit constants in Vinogradov's notation may depend on the discriminant~$D$.

\section{Local solubility}\label{local}
The goal of this section is to prove Proposition \ref{prop1}. We start by recalling a few results on local solubility obtained in \cite{JS}. 

\begin{lem}\label{local1}
Let $p\neq 2$ be some prime that is unramified in the field extension $\bbQ(\sqrt{D})$ and $A,B\in \bbZ$ such that $S^{(D;A,B)}$ is smooth. Then one has $S^{(D;A,B)}(\bbQ_p)\neq \emptyset$.
\end{lem}

This is part a) of Proposition 4.3 in \cite{JS}. Hence the only relevant primes are $2$, the infinite place and all ramified primes. Solubility over $\bbR$ is always guaranteed as for example noted in Remark 4.7 in \cite{JS}. Furthermore, if $p=2$ is split, then $S^{(D;A,B)}(\bbQ_2)\neq \emptyset$ by Lemma 4.4.a) in the same paper.\par
In the following, we set $G(A,B)=A^2-2AB+B^2-2A-2B+1$. We recall that the surface $S^{(D;A,B)}$ is smooth over $\bbQ$ if and only if $AB\neq 0$, $A\neq B$ and $G(A,B)\neq 0$ (see Proposition 2.1 in \cite{JS}). Note that if $S^{(D;A,B)}$ is smooth over $\bbQ$, then the same holds for all completions $\bbQ_p$. In the following, we give a more refined and quantitative version of this statement. We observe that if no high power of $p^{l+1}$ divides into any of the expressions $A,B$, $A-B$ or $G(A,B)$ and we are given a primitive solution modulo $p^{8l+1}$, then we can bound the multiplicity of the power of $p$ dividing all of the $2\times 2$ minors of the Jacobian by $4l$.

For convenience, we use in the following the vector notation $\bft=(t_0,\ldots,t_4)$ and set
\begin{equation*}
\begin{split}
Q_1(\bft)&= t_2^2-Dt_3^2-t_0t_1,\\
Q_2(\bft)&= t_2^2-Dt_4^2-(t_0+At_1)(t_0+Bt_1).
\end{split}
\end{equation*}
We also use vector notation for the system $\bfQ=(Q_1,Q_2)$ of quadratic forms.

\begin{lem}\label{local2}
Assume that $p\mid D$ and $p^2\nmid D$ where $p\neq 2$ is a prime. Let $l\geq 1$ be such that $p^{l+1}\nmid A,B,A-B,G(A,B)$. Assume that $\bft\in (\bbZ/p^{8l+1}\bbZ)^{5}$ has components not all divisible by $p$ and satisfies $\bfQ(\bft)\equiv 0 $ modulo $p^{8l+1}$.  Then $p^{4l+1}$ does not divide all $2\times 2$ minors of the Jacobian matrix $J(\bfQ)(\bft)$ at the point $\bft$.
\end{lem}

Note that the assumption $p^2\nmid D$ is crucial for the proof of Lemma \ref{local2}. However, since our discriminant $D$ is squarefree, this is no restriction in our application. 

\begin{proof}
Let $\bft\in (\bbZ/p^{8l+1}\bbZ)^{5}$ be as in the statement of the lemma. We first observe that $p\nmid t_1$. Otherwise, the congruence $Q_1(\bft)\equiv 0$ mod $p^2$ would imply that $p|t_2$ and then $p|(t_0+At_1)(t_0+Bt_1)$ by the second congruence $Q_2(\bft)\equiv 0$ mod $p^2$. This again gives $p|t_0$, which implies by the congruence from the first quadratic equation that $p|t_3$ and by the second that $p|t_4$, which is a contradiction to $\bft$ being primitive in a sense that not all of its coordinates are divisible by $p$.\par
Next we recall that the Jacobian matrix at the point $\bft$ is given by 
\begin{equation}\label{Jacobian}
\left(
\begin{array}{ccccc}
       t_1     &        t_0      & -2t_2 & 2Dt_3 &  0    \\
 2t_0+(A+B)t_1 & (A+B)t_0+2ABt_1 & -2t_2 &  0    & 2Dt_4 
\end{array}
\right).
\end{equation}
Assume that $p^{4l+1}$ divides all $2\times 2$ minors and that $\smash{\bfQ(\bft)\equiv 0 \mmod p^{8l+1}}$. We may already assume that $p\nmid t_1$. Since $p\neq 2$, we conclude first that $p^{4l+1}|t_4$. Furthermore, we observe that $p^{4l}|t_2t_3$ which implies that $p^{2l}|t_2$ or $p^{2l}|t_3$.\par
First assume that both $t_2,t_3$ are divisible by $p^{2l}$. In this case, the congruence $Q_1(\bft)\equiv 0 \mmod p^{8l+1}$ implies that $p^{4l}|t_0$ and the second quadratic congruence delivers $p^{4l}|AB$, which is a contradiction to our assumption.\par
In the case where $p^{2l}\nmid t_2$, we have that $p^{2l+1}|t_3$. Considering the $2\times 2$-minors of the Jacobian consisting of the 1st and 3rd and the 2nd and 3rd column, we find that
\begin{equation*}
  (A+B-1)t_1+2t_0\equiv (A+B-1)t_0+2ABt_1\equiv 0 \mmod p^{2l+2}.
\end{equation*}
This leads to 
\begin{equation*}
4ABt_1-(A+B-1)^2t_1\equiv 0 \mmod p^{2l+2},
\end{equation*}
and hence to $p^{2l+2}|G(A,B)$, which is a contradiction, as well.\par
Finally, let us consider the case where $p^{2l}\nmid t_3$ and $p^{2l+1}|t_2$. Computing the minors of the Jacobian consisting of the 1st and 4th and 2nd and 4th column shows that
\begin{equation*}
2t_0+(A+B)t_1\equiv (A+B)t_0+2ABt_1\equiv 0 \mmod p^{2l+1}.
\end{equation*}
This leads to 
\begin{equation*}
4ABt_1-(A+B)^2t_1\equiv 0 \mmod p^{2l+1},
\end{equation*}
and hence to $p^{l+1}|A-B$. 
\end{proof}

The restriction $p\neq 2$ is not strictly necessary in Lemma \ref{local2}, but one would need to change the exponents slightly for $p=2$. Since we assume $p=2$ to be split in our applications, we do not include this case into the lemma.\par

\begin{defi}
Let $\calR_l(p)$ be the set of residue classes of $A,B$ modulo $p^{8l+1}$ such that $ p^{l+1}\nmid A,B,A-B,G(A,B)$ and the congruence system $\bfQ(\bft)\equiv 0\mmod p^{8l+1}$ has a primitive solution.
\end{defi}

The following lemma justifies the definition of the sets $\calR_l(p)$ and explains their role. 

\begin{lem}\label{local3}
Let $p\neq 2$ be a ramified prime with $p^2\nmid D$ and $l\geq 1$. Assume that $p^{l+1}\nmid A,B,A-B,G(A,B)$. Then $S^{(D;A,B)}(\bbQ_p)\neq \emptyset$ if and only if $(A,B)$ modulo $p^{8l+1}$ is contained in $\calR_l(p)$.
\end{lem}

\begin{proof}
It is clear that $(A,B)\in \calR_l(p)$ if $S^{(D;A,B)}(\bbQ_p)\neq \emptyset$. Hence we need to show that there is a $\bbQ_p$-point on $S^{(D;A,B)}$ as soon as $(A,B)\in \calR_l(p)$. For this, recall that $p^{l+1}\nmid A,B,A-B,G(A,B)$. We assume that we are given a primitive vector $\bft \in (\bbZ/p^{8l+1}\bbZ)^{5}$ with $\bfQ(\bft)\equiv 0 \mmod p^{8l+1}$. By Lemma \ref{local2},  we know that $p^{4l+1}$ does not divide the determinants of all $2\times 2$ minors of the Jacobian $J(\bfQ)(\bft)$. Hence a version of Hensel's Lemma (see Proposition 5.21 in \cite{Greenberg}) implies that there is some $\bft'\in \bbZ_p^5$ such that $\bfQ(\bft')=0$ and $\bft'\equiv \bft$ modulo $p^{4l+1}$, and therefore $\bft'$ is in particular not the zero vector.
\end{proof}

We are now prepared to deduce the asymptotic for $R_D^\loc (N)$ as stated in Proposition \ref{prop1}. We note that the cases of $A,B$ for which $S^{(D;A,B)}$ is singular only contribute a small error. By Proposition 2.1 in \cite{JS}, the surface $S^{(D;A,B)}$ is singular if and only if $AB=0$ or $A-B=0$ or $G(A,B)=0$. And it is clear that
\begin{equation*}
\sharp \{ |A|,|B|\leq N: AB(A-B)G(A,B)=0\}\ll N.
\end{equation*}
We assume that $D\equiv 1 \mmod 8$ is squarefree. Then Lemma \ref{local1} implies that
\begin{equation*}
R_D^\loc(N)= \sharp\{|A|,|B|\leq N: S^{(D;A,B)}(\bbQ_p)\neq \emptyset\ \forall \ p|D\} +O(N).
\end{equation*}
Note that we always have $S^{(D;A,B)}(\R)\neq \emptyset$, since $D$ is positive.\par
Now we use the characterisation in Lemma \ref{local3} to detect local solubility at primes dividing $D$. For this, let $\smash{D=\prod_{i=1}^r p_i}$ be the prime factorization of $D$ into primes $p_1<\ldots <p_r$, and $L$ be the largest positive integer such that $\smash{\frac{N^{1/2}}{D^8}<D^{8L+1}\leq N^{1/2}}$. Then we have
\begin{equation*}
R_D^\loc (N)=\sharp \{|A|,|B|\leq N: (A,B)\mmod p_i^{8L+1}\in \calR_L(p_i)\ \forall \ 1\leq i\leq r\} + O(N) +E_1,
\end{equation*}
where $E_1$ is an error term bounded by
\begin{equation*}
E_1\ll \sum_{i=1}^r \sharp\{|A|,|B|\leq N: p_i^{L+1} \mbox{ divides one of } A,B,A-B,G(A,B)\}.
\end{equation*}
We observe that 
\begin{equation*}
8L+1 \geq \frac{1}{2}\frac{\log N}{\log D}-8.
\end{equation*}
Hence we have
\begin{equation*}
E_1\ll \sum_{i=1}^r \frac{N^2}{p_i^{L/2}} \ll \frac{N^2}{p_1^{L/2}} \ll \frac{N^2}{2^{L/2}}\ll N^{2-\tet(D)},
\end{equation*}
where $0<\tet(D)<1/2$ is given by $\tet(D)= \frac{\log 2}{32\log D}$.\par
We further rewrite the counting function $R_D^\loc (N)$ as 
\begin{equation*}
\begin{split}
R_D^\loc (N)= &\frac{4N^2}{(D^{8L+1})^2}\sharp\{A,B \mmod D^{8L+1}: (A,B)\in \calR_L(p_i) \ \forall \ 1\leq i\leq r\} \\ &+O(D^{8L+1}N) +O(N^{2-\tet(D)}).
\end{split}
\end{equation*}
We set $\calR_0(p):=\emptyset$ for all primes $p$. For any $l\geq 1$, we let $\calR_l^*(p)\subset \calR_l(p)$ be the set of tuples $(A,B)$ modulo $p^{8l+1}$ such that the reduction of $(A,B)$ modulo $p^{8(l-1)+1}$ is not contained in $\calR_{l-1}(p)$.
For each prime dividing $D$, we now sort the tuples $(A,B)$ according to the smallest $l$, for which $(A,B)\mmod p^{8l+1}\in \calR_l(p)$. In this way, we obtain
\begin{equation}\label{eqnlocal4}
\begin{split}
R_D^\loc (N)& =  \frac{4N^2}{D^{16L+2}}\sum_{l_1,\ldots, l_r=1}^L \sharp \{ A,B \mmod D^{8L+1}: (A,B)\mmod p_i^{8l_i+1}\in \calR_{l_i}^*(p_i)\ \forall 1\leq i\leq r\} \\  &\hspace{8cm}+O(N^{2-\tet(D)})\\ &=  \frac{4N^2}{D^{16L+2}}\sum_{l_1,\ldots, l_r=1}^L \prod_{i=1}^r\sharp \{ A,B \mmod p_i^{8L+1}: (A,B)\mmod p_i^{8l_i+1}\in \calR_{l_i}^*(p_i)\} +O(N^{2-\tet(D)})\\ &=  \frac{4N^2}{D^{16L+2}}\sum_{l_1,\ldots, l_r=1}^L \prod_{i=1}^r \left(\frac{p_i^{8L+1}}{p_i^{8l_i+1}}\right)^2\sharp \{ A,B \mmod p_i^{8l_i+1}: (A,B)\in \calR_{l_i}^*(p_i)\} \\ &\hspace{8cm}+O(N^{2-\tet(D)}) \\ &= 4N^2 \sum_{l_1,\ldots, l_r=1}^L \prod_{i=1}^r \frac{|\calR_{l_i}^*(p_i)|}{p_i^{2(8l_i+1)}} +O(N^{2-\tet(D)}).
\end{split}
\end{equation}
We claim that the last sum is absolutely convergent for $L\rightarrow \infty$. For this, we first observe that 
\begin{equation}\label{eqnx}
\begin{split}
|\calR_{l}^*(p)| &\ll \sharp \{ A,B \mmod p^{8l+1}: p^l \mbox{ divides one of } A,B, A-B, G(A,B)\} \\ &\ll \frac{p^{2(8l+1)}}{p^{\lfloor l/2\rfloor}}\ll_D \frac{p^{2(8l+1)}}{p^{l/2}}.
\end{split}
\end{equation}
Here we have used that $G(A,B)$ is a quadratic polynomial in $A,B$ and $p^{1/2}\ll_D 1$. 
Hence we can estimate
\begin{equation}\label{eqnlocal5}
\sum_{\substack{l_1,\ldots, l_r\\ \max (l_1,\ldots ,l_r)>L}} \hspace{-0.3cm}\prod_{i=1}^r \frac{|\calR_{l_i}^*(p_i)|}{p_i^{2(8l_i+1)}} \ll \hspace{-0.5cm}\sum_{\substack{l_1,\ldots, l_r\\ \max (l_1,\ldots ,l_r)>L}} \hspace{-0.3cm}\prod_{i=1}^r \frac{1}{p_i^{l_i/2}} \ll \sum_{l=L}^\infty \frac{1}{2^{l/2}} \ll 2^{-L/2}\ll N^{-\tet(D)}.
\end{equation}
Finally, we put 
\begin{equation}\label{cloc}
c_\loc := 4\sum_{l_1,\ldots, l_r=1}^\infty \prod_{i=1}^r \frac{|\calR_{l_i}^*(p_i)|}{p_i^{2(8l_i+1)}}= 4 \prod_{i=1}^r \sum_{l_i=1}^\infty \frac{|\calR_{l_i}^*(p_i)|}{p_i^{2(8l_i+1)}} .
\end{equation}
Then equation (\ref{eqnlocal4}) together with equation (\ref{eqnlocal5}) gives
\begin{equation*}
R_D^\loc (N)= c_\loc N^2+O(N^{2-\tet(D)}),
\end{equation*}
which proves the asymptotic in Proposition \ref{prop1} for some constant $c_\loc$. Next, we observe that the constant $c_\loc$ is indeed positive.

\begin{lem}
Let $D\equiv 1$ modulo $8$ and assume that $D$ is squarefree. Then one has the lower bound
\begin{equation*}
c_\loc \geq \frac{4}{D^2}.
\end{equation*}
\end{lem}

\begin{proof}
We use the expression for $c_\loc$ in (\ref{cloc}) to prove the lower bound $\smash{c_\loc \geq \frac{4}{D^2}}$. Note that $p|D$ implies that $p\neq 2$ by the congruence condition on $D$ modulo $8$. We first consider the case $p>3$. For this, we fix a choice of residue classes $(a,b)$ modulo $p$ with the property that $\smash{(\frac{a}{p})=1}$ and $a\not\equiv 0,-1$ modulo $p$ as well as $\smash{a^2+a+1\not\equiv 0}$ modulo $p$, and set $\smash{b\equiv \frac{a}{a+1}}$ modulo $p$. Such a choice is possible, since $\smash{p>3}$. If $\smash{(A,B)}$ is a pair of residue classes modulo some power $\smash{p^k}$ with $\smash{k\geq 1}$ that reduces to $(a,b)$ modulo $p$, then Proposition 5.1.a) in \cite{JS} shows that the system $\bfQ(\bft)=0$ has a primitive solution modulo $p^k$. We deduce that 
\begin{equation*}
\begin{split}
\sum_{l=1}^L &\frac{|\calR_l^*(p)|}{p^{2(8l+1)}}\geq \\ &p^{-2(8L+1)}  \sharp\{ (A,B) \mmod p^{8L+1}: (A,B)\mmod p = (a,b),\ p^{L+1}\nmid A,B,A-B,G(A,B)\}.
\end{split}
\end{equation*}
We take the limit for $L\rightarrow \infty$ and, in combination with the bound in (\ref{eqnx}), we obtain
\begin{equation}\label{eqny}
\sum_{l=1}^\infty \frac{|\calR_l^*(p)|}{p^{2(8l+1)}}\geq \frac{1}{p^2}.
\end{equation}
Now consider the case where $p=3$ and $p|D$. Then we choose $(a,b)=(0,0)$ and observe that $(1:1:1:0:0)$ is a smooth point on the reduction of $S^{(D;A,B)}$ for any $(A,B)$ that reduces to $(a,b)$ modulo $3$. Hence Hensel's Lemma implies that $S^{(D;A,B)}(\Q_3)\neq \emptyset$ for such $(A,B)$. Now the same argument as above shows that (\ref{eqny}) also holds for $p=3$. Together with equation (\ref{cloc}), this completes the proof of the lemma.
\end{proof}

\section{Evaluation of the Brauer group at inert primes}\label{section3}

For a surface $S^{(D;A,B)}$ in the family (\ref{eqn0}), we can explicitly write down a Brauer class, which is locally defined by one of the quotients $t_0/(t_0+At_1)$, $t_1/(t_0+At_1)$, $t_0/(t_0+Bt_1)$ or $t_1/(t_0+Bt_1)$.
Let $l$ be some place and $\bft\in S^{(D;A,B)}(\Q_l)$ a point, where one of the quotients is defined and non-zero. Denote one of the quotients by $q$. Then the evaluation of the Brauer class $\alp$  described in Proposition 3.2 in \cite{JS} is given by
\begin{equation*}
\ev_{\alp,l}(\bft)=\left\{\begin{array}{ccl}0 & \mbox{ if } & (q,D)_l=1, \\ \frac{1}{2} & \mbox{ if } & (q,D)_l=-1,\end{array}\right.
\end{equation*}
and the evaluation is independent of the choice of $q$ above.\par
In Proposition 4.3 in \cite{JS}, we observed that $S^{(D;A,B)}(\Q_p)\neq \emptyset$ as soon as $p\neq 2$ is a finite unramified prime. We even have the stronger statement that in this case there is always a point in $S^{(D;A,B)}(\Q_p)$ on which the Brauer class $\alp$ evaluates to $0$. 

\begin{lem}\label{eval0}
let $p\neq 2$ be some unramified prime in the extension $\Q(\sqrt{D})$. Then there is a point $\bft\in S^{(D;A,B)}(\Q_p)$ such that $\ev_{\alp,p}(\bft)=0$. 
\end{lem}

\begin{proof}
In the case where $p$ is split, this is clear and only requires the existence of some point $\bft\in S^{(D;A,B)}(\Q_p)$, which is guaranteed by Proposition 4.3 in \cite{JS}.\par
Let $p\neq 2$ be some inert prime. In the proof of Proposition 4.3 in \cite{JS}, we showed that there is a regular $\F_{\! p}$-rational point on the reduction of $S^{(D;A,B)}$. Considering the Jacobian (\ref{Jacobian}) at this point together with the system of equations defining $S^{(D;A,B)}$, we see that, for each $(t_0:\ldots:t_4)\in S^{(D;A,B)}(\Q_p)$ lifting it, at least one of $t_0,t_1$ and one of $t_0+At_1$, $t_0+Bt_1$ has to be a unit. The corresponding quotient $q$ then satisfies $(q,D)_p=1$ and hence $\ev_{\alp,p}(\bft)=0$.
\end{proof}

For an inert prime $p\neq 2$, we hence need to distinguish two cases. Either $S^{(D;A,B)}(\Q_p)\neq \emptyset$ and the Brauer class evaluates constantly to $0$, or there are $\Q_p$-rational points, but $\alp$ takes both values $0$ and $1/2$ on $S^{(D;A,B)}(\Q_p)$. We give some criteria for both cases in the next lemma. Let $\nu_p$ be the $p$-adic valuation on $\Q_p$. 

\begin{lem}\label{inert}
Let $p\neq 2$ be some inert prime and $\alp$ the Brauer class described above. Assume that $\nu_p(A)\leq \nu_p(B)$.
~\vspace{0.2cm}
\newline
i) If $\nu_p(A)$ is odd, then the evaluation of $\alp$ on $S^{(D;A,B)}(\Q_p)$ is constant if and only if $B$ is a square in $\Q_p$.\vspace{0.2cm}\\
\smallskip
ii) If $\nu_p(A)$ is even, then the evaluation of $\alp$ is non-constant if and only if $\nu_p(B-A)>\nu_p(A)$ and $BD$ is a square.\vspace{0.2cm}\\
In the case of constancy, the Brauer class takes the value $0$ on all of $S^{(D;A,B)}(\Q_p)$.
\end{lem}


Note that, by symmetry, Lemma \ref{inert} covers all cases of choices for integers $A$ and $B$. 

\begin{proof}
By Lemma \ref{eval0}, we already know that $\smash{S^{(D;A,B)}(\Q_p)\neq \emptyset}$ and that $\alp$ takes the value $0$ on some element in this set. Let $\smash{\bft\in S^{(D;A,B)}(\Z_p)}$ be a primitive solution, i.e.~one such that not all of the coordinates of $\bft$ are divisible by $p$. If $t_0$ and $t_1$ were both divisible by $p$, then also $t_2$ and $t_3$ by the first of the two equations of (\ref{eqn0}), and hence also $t_4$ by the second equation, which is a contradiction to the primitivity of the solution. Hence one of $t_0$ or $t_1$ is a unit. If $t_0$ is a unit and $t_1$ is divisible by $p$, then both of the factors $t_0+At_1$ and $t_0+Bt_1$ have even valuation, and hence $\ev_{\alp,p}(\bft)=0$. Therefore, the only points of interest to us are those where $t_1$ is a unit. Furthermore we note that the first equation in (\ref{eqn0}) implies that $t_0$ has even $p$-adic valuation. In the cases of non-constancy of the lemma, we need to find some element $\bft\in S^{(D;A,B)}(\Q_p)$ with $\nu_p(t_0+At_1)$ and $\nu_p(t_0+Bt_1)$ both being odd, and in the other cases we need to show that any primitive solution $\bft\in S^{(D;A,B)}(\Z_p)$ with $t_1$ a unit has the property that $\nu_p(t_0+At_1)$ and $\nu_p(t_0+Bt_1)$ are even. By homogeneity, we may in this case even assume that $t_1=1$. We prove the lemma in three steps, where we distinguish different cases (which are again different than in the formulation of the lemma).\vspace{0.2cm}\\
{\em First step:} We claim that if $\nu_p(A)$ and $\nu_p(B)$ are both odd, then the evaluation of $\alp$ on $S^{(D;A,B)}(\Q_p)$ is non-constant.\par
For this, let $\smash{a=\nu_p(A)}$ and $\smash{b=\nu_p(B)}$, and write $\smash{A=p^a u}$ and $\smash{B=p^b v}$ with units $u$ and $v$. We use the substitution $\smash{t_0=p^{a+b}y_0}$ and $t_1=y_1$ and $\smash{t_i=p^{\frac{a+b}{2}}y_i}$ for $i=2,3,4$. Then the system of equations (\ref{eqn0}) simplifies to 
\begin{equation*}
\begin{split}
y_0y_1&=y_2^2-Dy_3^2,\\
(p^by_0+uy_1)(p^ay_0+vy_1)&=y_2^2-Dy_4^2.
\end{split}
\end{equation*}
The reduction modulo $p$ is given by 
\begin{equation*}
\begin{split}
y_0y_1&=y_2^2-Dy_3^2,\\
uvy_1^2&=y_2^2-Dy_4^2.
\end{split}
\end{equation*}
We put $y_1=1$ in $\F_{\! p}$ and solve the second equation in $y_2,y_4$ over the finite field $\F_{\! p}$. Then we can solve the first equation in $y_0$ after choosing some arbitrary $y_3$. The Jacobian at this point has full rank since $y_1\neq 0$ and $2uv$ is a unit. Hence this solution lifts to a solution in $\Z_p$. The corresponding point $\bft$ has the property that $t_1$ has even and $t_0+At_1$ has odd valuation, and hence $\ev_{\alp,p}(\bft)=1/2$, as desired.\vspace{0.2cm}\\
{\em Second step:} Let now $\nu_p(A)$ be even and $\nu_p(A)=\nu_p(B-A)$. We then claim that the evaluation on $\smash{S^{(D;A,B)}(\Q_p)}$ is constantly zero.\par
As we noted at the beginning of the proof, it is sufficient to show that all reduced vectors $\bft\in S^{(D;A,B)}(\Z_p)$ with $t_1=1$ have the property that $\ev_{\alp,p}(\bft)=0$. Assume, to the contrary, that both $t_0+A$ and $t_0+B$ have odd $p$-adic evaluation. 
If we keep the notation $a=\nu_p(A)$, then we see that $t_0=-A+rp^{a+1}$ for some $r\in \Z_p$. However, then the term $t_0+B=B-A+rp^{a+1}$ has even valuation, a contradiction. Hence we conclude that the evaluation of the Brauer class $\alp$ on $S^{(D;A,B)}(\Q_p)$ is constant.\vspace{0.2cm}\\
{\em Third step:} Assume that $\nu_p(A)<\nu_p(B)$, and additionally that $\nu_p(A)$ is odd and $\nu_p(B)$ even. Or that $\nu_p(A) =\nu_p(B)$ are even and $\nu_p(B-A)>\nu_p(B)$.\par
We first aim to show that $\alp$ evaluates constantly in the case where $BD$ is not a square in~$\Q_p$. For this, it suffices to consider a primitive solution $\bft\in S^{(D;A,B)}(\Z_p)$ with $t_1=1$. If one of $t_0+A$ or $t_0+B$ has even valuation, then $\alp$ evaluates to zero at this point. Hence, we may assume that $t_0+A$ and $t_0+B$ both have odd $p$-adic valuation. Since $B$ has in both cases even $p$-adic valuation, we conclude that $t_0=-B+sp^{b+1}$ for some $s\in \Z_p$. For such a solution, we have $t_0+A=A-B+sp^{b+1}$. The second equation in (\ref{eqn0}) shows that $p^{b+2}$ divides $t_2^2-Dt_4^2$ and hence $\smash{p^{\frac{b+2}{2}}}$ divides both $t_2$ and $t_4$. The first equation in (\ref{eqn0}) simplifies to 
\begin{equation*}
-B+sp^{b+1}=t_2^2-Dt_3^2.
\end{equation*}
Since the $p$-adic valuation of $t_2^2$ is at least $b+2$, we observe that $BD$ must be a square in $\Q_p$. This is a contradiction to our assumption, and hence we have shown that $\alp$ evaluates constantly to zero on $S^{(D;A,B)}(\Q_p)$ in the case where $BD$ is a non-square.\par
We now claim that $\alp$ evaluates non-constantly if $BD$ is a square in $\Q_p$. For this we construct solutions $\bft\in S^{(D;A,B)}(\Q_p)$ with $\ev_{\alp,p}(\bft)=1/2$. We hence assume that $BD$ is a square in $\Q_p$ and then distinguish two subcases. In the first subcase, we assume that $\nu_p(B-A)$ is odd. We put $t_1=1$, $t_2=0$ and 
set $t_0=-B+sp^{b+1}$ for some $s\in \Z_p$ to be chosen later. The second equation of (\ref{eqn0}) simplifies to
\begin{equation*}
sp^{b+1}(A-B+sp^{b+1})=-Dt_4^2.
\end{equation*}
We now choose $s\in \Z_p$ of even $p$-adic valuation in a way that $\nu_p(A-B)<\nu_p(s)+b+1$ and such that $-Dsp^{b+1}(A-B)$ is a square in $\Q_p$. This is possible since $b+1$ and $\nu_p(A-B)$ are both odd. Then we can solve the second equation for $t_4$. The first equation of (\ref{eqn0}) simplifies to 
\begin{equation*}
-B+sp^{b+1}=-Dt_3^2.
\end{equation*}
Since we have assumed that $BD$ is a square in $\Q_p$, the same is true for $BD-Dsp^{b+1}$, and hence we can solve for $t_3$. Our constructed point $\bft\in S^{(D;A,B)}(\Z_p)$ has the property that $t_0+B$ has odd $p$-adic valuation and $t_1=1$, and hence $\ev_{\alp,p}(\bft)=1/2$.\par
For the last subcase that $\nu_p(B-A)$ is even, recall that we have that $\nu_p(A) =\nu_p(B)$ are even and $\nu_p(B-A)>\nu_p(B)$. We construct a point $\bft\in S^{(D;A,B)}(\Q_p)$ with $\ev_{\alp,p}(\bft)=1/2$ in the following way. Let $b=\nu_p(B)$. We set $t_1=1$ and $t_4=0$. Furthermore, let $t_0=-B+vp^{b+1}$ with $v$ a unit to be chosen later. Note that $\nu_p(B-A)\geq b+2$. The second equation of (\ref{eqn0}) then simplifies to 
\begin{equation*}
(A-B+vp^{b+1})(vp^{b+1})=t_2^2.
\end{equation*}
We can solve this for $t_2$ since $v^2p^{2(b+1)}$ is a square in $\Q_p$ and hence also the left hand side of the equation. Note that, in particular, we obtain that $p^{b+1}|t_2$. It remains to consider the first equation in (\ref{eqn0}), which simplifies to
\begin{equation*}
-B+vp^{b+1}=t_2^2-Dt_3^2.
\end{equation*}
Since $t_2$ is divisible by $p^{b+1}$, this is soluble for $t_3$ if and only if $BD$ is a square in $\Q_p$, which is satisfied by our assumption. We conclude that our constructed point $\bft\in S^{(D;A,B)}(\Q_p)$ satisfies $t_1=1$ and $t_0+B=vp^{b+1}$, which has odd $p$-adic valuation. Hence $\ev_{\alp,p}(\bft)=1/2$, as desired.
\end{proof}

\section{First asymptotics}\label{section4}

As before, let $D=\prod_{i=1}^r p_i$ be a factorization of $D$ into distinct primes $p_1<\ldots < p_r$. In this section, we fix some modulus $T$, which is composed of primes dividing the discriminant $D$, and two congruence classes $a$ and $b$ modulo $T$. We seek density estimates for the number of surfaces $S^{(D;A,B)}$ in the family (\ref{eqn0}) with $(A\mmod T) =a $ and $(B\mmod T)=b$ that are counterexamples to the Hasse principle explained by some algebraic Brauer-Manin obstruction. For this, we introduce the counting function $\RBR (N;T,a,b)$, which counts the number of $|A|,|B|\leq N$ with $(A\mmod T) =a $ and $(B\mmod T)=b$ such that $S^{(D;A,B)}$ is a counterexample to the Hasse principle explained by some Brauer-Manin obstruction. 

\begin{lem}\label{Brauer}
Assume that $S^{(D;A,B)}$ is non-singular, has an adelic point and that neither of the expressions $-AB$ or $D((A+B-1)^2-4AB)$ is a square in $\bbQ$. Then $\Br (S^{(D;A,B)})/\Br (\Q)$ is isomorphic to $\bbZ/2\bbZ$ or $0$.
\end{lem}

\begin{proof}
Generally, the Brauer group of $S^{(D;A,B)}$ can be either isomorphic to $0$ or $\Z/2\Z$ or $(\Z/2\Z)^2$. Let $\calS\subset \bbP^1$ be the degeneracy locus of the pencil of the two quadratic forms defining $S^{(D;A,B)}$. In particular, $\calS$ is a degree five subscheme of $\bbP^1$. Since $S^{(D;A,B)}(\A_\Q)\neq \emptyset$, we may apply Theorem 3.4 in \cite{VAV} (see also \cite{CTSSD} and \cite{Wi1}). This includes the statement that $\Br (S^{(D;A,B)})/\Br (\Q)\cong (\Z/2\Z)^2$ if and only if $\calS$ has three distrinct points $s_0,s_1,s_2\in \calS(\Q)$ such that the corresponding discriminants $D_{s_0},D_{s_1},D_{s_2}$ of the rank four quadrics are non-squares in $\Q$ and coincide up to square factors.\par
Hence let us compute the characteristic polynomial $P(\lam,\mu)=\det (\lam Q_1+\mu Q_2)$ for $Q_1$ and $Q_2$ the two quadratic forms in (\ref{eqn0}) and obtain
\begin{equation*}
\begin{split}
P(\lam,\mu)&= (AB\mu^2-\frac{1}{4}(\lam+\mu(A+B))^2)(\lam+\mu)(-D\lam)(-D\mu)\\
&=-\frac{1}{4}D^2\left(\mu^2(A-B)^2+2\lam\mu (A+B)+\lam^2\right)(\lam+\mu)\lam \mu.
\end{split}
\end{equation*}
The two points of $\calS$ corresponding to the quadratic factor are defined over $\Q$ if and only if the discriminant of the quadratic form $\mu^2(A-B)^2+2\lam\mu (A+B)+\lam^2$ is a square in $\Q$. This is given by 
\begin{equation*}
-(2^2(A+B)^2-4(A-B)^2)=-16AB,
\end{equation*}
and hence equals $-AB$ up to square factors. We conclude that under the assumptions of the lemma, exactly three points $s_0,s_1,s_2$ of $\calS$ are defined over $\Q$. The corresponding rank four quadrics are those in (\ref{eqn0}) and the quadric
\begin{equation*}
t_0^2+(A+B-1)t_0t_1+ABt_1^2=Dt_3^2-Dt_4^2.
\end{equation*}
The corresponding discriminants are given up to square factors by $D_{s_0}=D$, $D_{s_1}=D$ and 
\begin{equation*}
D_{s_2}=-D^2\det\left(\begin{array}{cc} 1 & \frac{1}{2}(A+B-1)\\ \frac{1}{2}(A+B-1) &AB\end{array}\right)= \frac{1}{4}D^2((A+B-1)^2-4AB).
\end{equation*}
By the assumption of the lemma, the discriminant $D_{s_2}$ does not coincide with $D_{s_0}$ or $D_{s_1}$ up to square factors, and hence Theorem 3.4 in \cite{VAV} implies that the Brauer group cannot be isomorphic to $(\Z/2\Z)^2$.
\end{proof}


We claim that the contribution of those $A$ and $B$, for which Lemma \ref{Brauer} does not apply, is negligible. 

\begin{lem}\label{lem3}
Let $\calQ$ be the set of squares in $\Q$. One has the bounds
\begin{equation*}
\sharp\{|A|,|B|\leq N: -AB\in \calQ\}\ll_\eps N^{1+\eps},
\end{equation*}
and
\begin{equation*}
\sharp\{|A|,|B|\leq N: D((A+B-1)^2-4AB)\in \calQ\}\ll_\eps D^\eps N^{1+\eps}.
\end{equation*}
\end{lem}

\begin{proof}
The first estimate is clear since
\begin{equation*}
\sharp\{|A|,|B|\leq N: -AB\in \calQ\}\ll \sum_{|z|\leq N}d(z^2)\ll_\eps N^{1+\eps}.
\end{equation*}
For the second bound in the lemma, we rewrite the quadratic form 
\begin{equation*}
(A+B-1)^2-4AB= (A-B)^2-2(A+B)+1.
\end{equation*}
Hence, the second counting function in the lemma is bounded by
\begin{equation}\label{lem3eqn1}
\begin{split}
\sharp\{|A|,|B|\leq N: D((A+B-1)^2-4AB)\in \calQ\}&\ll \sharp\{|u|,|v|\leq 4N+1:\ D(u^2-v)=x^2\}\\ &\ll \!\!\sum_{|v|\leq 4N+1} \!\!\!\!\!\sharp\{ |u|\leq 4N+1: -Dv=x^2-Du^2\}.
\end{split}
\end{equation}
We note that $x^2-Du^2$ is a norm form and introduce the representation function
\begin{equation*}
\rho_B(n)=\sharp\{|u|,|x|\leq B:\ n=x^2-Du^2\}.
\end{equation*}
By Lemma 4.3 in \cite{Ple75}, we have the upper bound $\rho_B(n)\ll_\eps |n|^\eps B^\eps$. Hence we may now bound the counting function in (\ref{lem3eqn1}) by
\begin{equation*}
\sharp\{|A|,|B|\leq N: D((A+B-1)^2-4AB)\in \calQ\}\ll\!\! \sum_{|v|\leq (4N+1)}\!\!\!\!\!\!\rho_{8D(N+1)}(-Dv)\ll_\eps D^{\eps}N^{1+\eps},
\end{equation*}
which completes the proof of the lemma.
\end{proof}

Recall that, if $\smash{\Br (S^{(D;A,B)})/\Br(\Q)\cong \Z/2\Z}$ and $\smash{S^{(D;A,B)}}$ has an adelic point, then there is a Brauer-Manin obstruction to the Hasse principle if and only if the non-trivial Brauer class evaluates constantly at each place and takes the value $\smash{\frac{1}{2}}$ an odd number of times. We show next that, for $T$ sufficiently large, the surfaces showing this behaviour may be characterized entirely in terms of the residue classes $(A \mmod T)$ and $(B\mmod T)$.\par

\begin{lem}\label{lem1}
There is a set of non-zero polynomials $G_i\in \Z[X,Y]$, $1\leq i\leq m$, such that the following holds. Let $p\mid D$ be a prime not equal to $2$ with $p^2\nmid D$, and assume, for certain $A,B\in \Z$, that $p^l\nmid G_i(A,B)$ for all $1\leq i\leq m$.\vspace{0.2cm}\\
a) Then the local solubility $S^{(D;A,B)}(\Q_p)\neq \emptyset$ only depends on $A$ and $B$ modulo $p^l$.\vspace{0.2cm}\\
b) Furthermore, the set of values taken by the evaluation of $\alp$ on $S^{(D;A,B)}(\Q_p)$, as described in Section \ref{section3}, only depends on $A$ and $B$ modulo $p^l$.\vspace{0.2cm}\\
Moreover, one may take $m=4$ and $G_1(X,Y)=G(X,Y)^8= (X^2-2XY+Y^2-2X-2Y+1)^8$, $G_2(X,Y)=X^8$, $G_3(X,Y)=Y^8$ and $G_4(X,Y)=(X-Y)^8$.
\end{lem}

The proof of Lemma \ref{lem1} is similar to the proof of Lemma \ref{local2} and Lemma \ref{local3}.

\begin{proof}
Let $G_i(X,Y)$ for $1\leq i\leq 4$ be as chosen above, and assume that $p^l\nmid G_i(A,B)$ for $1\leq i\leq 4$. Lemma \ref{local3} now implies that local solubility of $S^{(D;A,B)}$ over $\Q_p$ only depends on $(A,B)$ modulo~$p^l$. This proves the first part of the lemma.\par
Next we need to understand the evaluation of $\alp$ on $S^{(D;A,B)}(\Q_p)$. For this let $\bft\in S^{(D;A,B)}(\Q_p)$ be a point, which we may assume to have coordinates in $\Z_p$ in reduced form. Then, as shown in the proof of Lemma \ref{local2}, we have that $\smash{p\nmid t_1}$. Furthermore, we claim that $\smash{p^l\nmid t_0+At_1}$ or $\smash{p^l\nmid t_0+Bt_1}$. Indeed, otherwise we would have $p^l| A-B$, which is a contradiction to $p^l\nmid (A-B)^8$. Hence the $p$-adic valuation of $t_1(t_0+At_1)$ or $t_1(t_0+Bt_1)$ is at most $l-1$. Therefore, the evaluation of $(q,D)_p$, with $q=t_1/(t_0+At_1)$ or $q=t_1/(t_0+Bt_1)$, only depends on $\bft$ and $A,B$ modulo $p^l$. Moreover, any primitive solution modulo $p^l$ lifts according to Lemma \ref{local2} and Hensel's Lemma (for example in the form of Proposition 5.21 in \cite{Greenberg}) to a solution in $S^{(D;A,B)}(\Q_p)$. In order to find all possible values of $\alp$ on $S^{(D;A,B)}(\Q_p)$, one hence only needs to consider $A$ and $B$ modulo $p^l$ and evaluate $\alp$ on all primitive solutions modulo $p^l$. The result only depends on $A$ and $B$ modulo $p^l$.
\end{proof}

In the following, we use the notation $G_j(X,Y)$, $1\leq j\leq 4$, for the four polynomials specified at the end of Lemma \ref{lem1}.\par
\medskip
{\bf Notation}. With the conclusions of Lemma \ref{lem1} in mind, for a vector $\bfl=(l_1,\ldots, l_r)\in \N^r$, we define $$\calH(\bfl)\subset \left(\Z/\prod_{i=1}^r p_i^{l_i}\Z\right)^2$$ to be the set of all pairs $(a,b)$ modulo $\prod_{i=1}^r p_i^{l_i}$ such that\\
a) for each $1\leq i\leq r$, there is some $1\leq j\leq m$ with $p_i^{l_i-1}|G_j(a,b)$,\\
b) for each $1\leq i\leq r$ and $1\leq j\leq m$, one has $p_i^{l_i}\nmid G_j(a,b)$,\\
c) one has $S^{(D;A,B)}(\Q_{p_i})\neq \emptyset$ for $(A\mmod \prod_{i=1}^r p_i^{l_i})=a$, $(B \mmod \prod_{i=1}^r p_i^{l_i})=b$, and all $1\leq i\leq r$, and\\
d) the Brauer class $\alp$ described in section \ref{section3} evaluates constantly at all places $p_i$ and takes the value $1/2$ at an odd number of them.\\

Before we state a lemma, which we use to characterize surfaces $S^{(D;A,B)}$ in our family (\ref{eqn0}) that are counterexamples to the Hasse principle explained by some Brauer-Manin obstruction, we give an easy upper bound for the cardinality of the set $\calH(\bfl)$.

\begin{lem}\label{lem1b}
There is a positive real constant $\tet_0$, such that
\begin{equation*}
\sharp\calH(\bfl)\ll_D \prod_{i=1}^r p_i^{2(l_i-\tet_0 )}.
\end{equation*}
More precisely, the bound is valid for any $\tet_0<1/16$. 
\end{lem}

\begin{proof}
It is enough to use property a) in the definition of the set $\calH(\bfl)$, and bound
\begin{equation*}
\sharp \calH(\bfl) \ll \prod_{i=1}^r \sharp \{(a,b)\mmod p_i^{l_i}: p_i^{l_i-1}|G_j(a,b) \mbox{ for some } 1\leq j\leq 4\}.\vspace{-0.7cm}
\end{equation*}
\end{proof}

\begin{lem}\label{lem2}
Let $\smash{D\equiv 1 \mmod 8}$ be squarefree with a factorization into primes $\smash{D=\prod_{i=1}^r p_i}$,  as before. Let $a$ and $b$ be congruence classes modulo $\smash{T=\prod_{i=1}^r p_i^{l_i}}$ such that $\smash{(a,b)\in \calH(\bfl)}$. Assume that $(A\mmod T)=a$ and $(B\mmod T)=b$ and that neither of the expressions $-AB$ or $D((A+B-1)^2-4AB)$ is a square in $\Q$.  Furthermore, assume that $S^{(D;A,B)}$ is non-singular.\par
Then there is a Brauer-Manin obstruction to the Hasse principle for $S^{(D;A,B)}$ if and only if, for all inert primes $q$ in $\bbQ(\sqrt{D})/\bbQ$, the evaluation of the Brauer class $\alp$, as described in section \ref{section3}, is constant.
\end{lem}

\begin{proof}
Note that the condition $D\equiv 1 \mmod 8$ ensures that $2$ is split in the quadratic extension $\bbQ(\sqrt{D})/\bbQ$. Hence $S^{(D;A,B)}(\Q_2)\neq \emptyset$ by Lemma 4.4 in \cite{JS}. Furthermore, the definition of $\calH(\bfl)$ ensures that $S^{(D;A,B)}(\Q_p)\neq \emptyset$ for all ramified primes and Lemma \ref{local1} ensures local solubility at all unramified primes different from $2$. Since $D>0$, it is clear that there are real solutions. Hence one has $\smash{S^{(D;A,B)}(\A_\Q)\neq \emptyset}$. If there is some inert prime $q$, for which the evaluation of $\alp$ on $\smash{S^{(D;A,B)}(\Q_q)}$ is non-constant, then this shows that $\alp$ defines a non-trivial element in $\smash{\Br (S^{(D;A,B)})/\Br(\Q)}$. By~Lemma \ref{Brauer} the element $\alp$ already generates $\Br (S^{(D;A,B)})/\Br(\Q)$ and hence there is no Brauer-Manin obstruction to the Hasse principle.\par
For the other direction, as $(a,b)\in \calH(\bfl)$, we have constant evaluation at all ramified primes, whereas the evaluation takes the value $\smash{\frac{1}{2}}$ an odd number of times. Moreover, we note that if $\alp$ evaluates constantly at some unramified prime different from $2$, then it automatically takes the value zero by Lemma \ref{eval0}. Also, the evaluation of $\alp$ at the prime $2$ is constantly zero, as this prime is split. Hence $S^{(D;A,B)}$ is a counterexample to the Hasse principle, explained by some Brauer-Manin obstruction, if $\alp$ evaluates constantly on $S^{(D;A,B)}(\Q_q)$ for all inert primes~$q$. 
\end{proof}

Before we start to establish an asymptotic formula for $\RBR (N;T,a,b)$, let us introduce the following definition.


\begin{defi}\label{admissible}
Let $n\neq 0$. We call an integer $B$ {\em admissible} for $n$, if the following two conditions hold.\vspace{0.1cm}\\
a) If $p$ is an inert prime and $p^l\|n$ for some odd $l$, then $B$ is of the form $B=up^{2k}$ with $\smash{(\frac{u}{p})=1}$, or $\smash{B=-n+up^{2k}}$ with $2k>l$ and $\smash{(\frac{u}{p})=1}$.\vspace{0.1cm}\\
b) If $p$ is an inert prime and $p^l\|n$ for some even $l$, then either $p^l|B$ or $B$ is of the form $B=up^{2k}$ with $2k<l$ and $(\frac{u}{p})=1$.\\
\end{defi}

We can now characterize elements in the family $S^{(D;A,B)}$, for which there is a Brauer-Manin obstruction to the Hasse principle.

\begin{lem}\label{lem2b}
Let $a$ and $b$ be congruence classes modulo $T$ such that $(a,b)\in \calH(\bfl)$. Assume that $(A \ {\rm mod} \ T)=a$, $(B\ {\rm mod}\ T)=b$, and that $S^{(D;A,B)}$ is non-singular and neither of the expressions $-AB$ or $D((A+B-1)^2-4AB)$ is a square in $\Q$. Put $n:=A-B$.\par
Then there is a Brauer-Manin obstruction to the Hasse principle for $S^{(D;A,B)}$ if and only if $B$ is admissible for $n$. 
\end{lem}

\begin{proof}
By Lemma \ref{lem2}, we need to show that admissibility is equivalent to saying that $\alp$ evaluates constantly to zero at all inert primes $p$. For this, we consider some fixed inert prime $p$.\par
We have $A=n+B$ and it is, of course, possible that $\nu_p(A)<\nu_p(B)$. This happens if and only if $\nu_p(n)<\nu_p(B)$. In this case, $\nu_p(n)=\nu_p(A)$. Thus, Lemma \ref{inert} shows that the evaluation of $\alp$ is constant at the prime $p$ if and only if we are in one of the four cases below.\vspace{0.1cm}\\
i) $\nu_p(n)<\nu_p(B)$, $\nu_p(n)$ is odd, and $B$ is a square.\vspace{0.1cm}\\
ii) $\nu_p(n)<\nu_p(B)$ and $\nu_p(n)$ is even.\vspace{0.1cm}\\
iii) $\nu_p(n)\geq \nu_p(B)$, $\nu_p(B)$ is odd, and $A$ is a square.\vspace{0.1cm}\\
As $A=n+B$ and $\nu_p(B)$ is odd, the latter is possible only when $\nu_p(n)=\nu_p(B)$. I.e., if $B=-n+q$ for $q$ a square such that $\nu_p(q)>\nu_p(n)$.\vspace{0.1cm}\\
iv) $\nu_p(n)\geq \nu_p(B)$, $\nu_p(B)$ is even, and $\nu_p(n)=\nu_p(B)$ or $AD$ is a non-square. The last statement is hence of interest only when $\nu_p(n)>\nu_p(B)$. In which case, $AD= (n+B)D$ being a non-square is equivalent to $BD$ being a non-square, and to $B$ being a square.\vspace{0.2cm}\par
Thus, given $n$, $\alp$ evaluates constantly at the prime $p$ if and only if one of the following holds.\vspace{0.1cm}\\
$\bullet$ $\nu_p(n)>\nu_p(B)$ and $B$ is a square.\vspace{0.1cm}\\
$\bullet$ $\nu_p(n)$ is odd, $\nu_p(n)<\nu_p(B)$ and $B$ is a square.\vspace{0.1cm}\\
$\bullet$ $\nu_p(n)$ is odd, and $B=-n+q$, for $q$ a square such that $\nu_p(q)>\nu_p(n)$.\vspace{0.1cm}\\
$\bullet$ $\nu_p(n)$ is even and $\nu_p(n)\leq \nu_p(B)$.\vspace{0.1cm}\\
In view of Definition \ref{admissible}, this completes the proof.
\end{proof}

We now define the counting function
\begin{equation*}
r(N,n)=\sharp\{|B|\leq N: (B\mmod T)=b,\ |B+n|\leq N,\ B\mbox{ is admissible for } n\},
\end{equation*}
and for convenience of notation also write $r(N,n)=r(n)$ if the dependence on $N$ is clear. Let $(a,b)\in \calH (\bfl)$. By the above considerations, we can rewrite the counting function $\RBR (N;T,a,b)$ as
\begin{equation*}
\RBR (N;T,a,b)= \!\!\sum_{\substack{n\equiv a-b \mmod T\\ |n|\leq 2N}}\!\! r(n) + O(E_1)+O(E_2)+O(E_3),
\end{equation*}
with error terms of the form
\begin{equation*}
E_1=\sharp\{|A|,|B|\leq N: (A\mmod T)=a, (B\mmod T)=b, S^{(D;A,B)} \mbox{ is singular}\},
\end{equation*}
\begin{equation*}
E_2=\sharp\{|A|,|B|\leq N: -AB\in \calQ\},
\end{equation*}
and
\begin{equation*}
E_3=\sharp\{|A|,|B|\leq N: D((A+B-1)^2-4AB)\in \calQ\}.
\end{equation*}
By Lemma \ref{lem3}, we have
\begin{equation*}
E_2\ll N^{1+\eps},\quad \mbox{ and } \quad E_3\ll D^\eps N^{1+\eps}.
\end{equation*}
Next, we note that the set of $A$ and $B$ such that $S^{(D;A,B)}$ is singular is rather sparse and will give a negligible contribution. By Proposition 2.1 in \cite{JS}, we have
\begin{equation*}
E_1\ll\sharp\{|A|,|B|\leq N: A=B, \mbox{ or } AB=0, \mbox{ or } A^2-2AB+B^2-2A-2B+1=0\}. 
\end{equation*}
Hence, we see that $E_1\ll N$ and 
\begin{equation}\label{eqn4.10}
\RBR (N;T,a,b)= \!\!\sum_{\substack{n\equiv a-b \mmod T\\ |n|\leq 2N}} \!\!\!\!\!\!\!\! r(n) +O_\eps(D^\eps N^{1+\eps}).
\end{equation}
Note that the implied constant in the error term is independent of $T$. In our computations, we will generally keep explicit dependence of the error terms on $T$, whereas the implicit constants may depend on $D$.\par
Our next goal is to approximate the function $r(n)$ by some linear combination of multiplicative functions, which then can be used to evaluate the main term in the asymptotic for $\RBR (N;T,a,b)$. For this, we introduce the multiplicative function $\sig(m)$ for $m\in \N$, which is defined in the following way. If $l$ is an even positive integer and $p$ some inert prime, then we put
\begin{equation*}
\sig(p^l):=\frac{1}{p^l} +\sum_{k=0}^{\frac{l}{2}-1}\frac{p-1}{2p^{2k+1}},
\end{equation*}
and note that
\begin{equation*}
\sig(p^l)= \frac{1}{p^l}+ \frac{(1-p^{-l})}{2(1+p^{-1})}.
\end{equation*}
For $l$ odd and $p$ an inert prime, we set
\begin{equation*}
\sig(p^l):=  \sum_{k=0}^\infty \frac{p-1}{2p^{2k+1}} +\!\!\!\!
                 \sum_{k=(l+1)/2}^\infty \frac{p-1}{2p^{2k+1}}=\frac{1+p^{-(l+1)}}{2(1+p^{-1})}.
\end{equation*}
We extend $\sig$ to a multiplicative function on all of $\Z$ by setting $\sig(m):=1$ if $m$ is not divisible by any inert prime, and $\sig(-1):=1$.\par

\begin{lem}\label{lem4.4}
Let $q_1,\ldots, q_\tau$ be the list of the inert primes dividing $n$. One has
\begin{equation*}
r(n)= \frac{2N-|n|+1}{T} \sig(n)+\ra (n),
\end{equation*}
with an error $\ra (n)$, which is absolutely bounded by
\begin{equation*}
\ra (n)\ll \left(\prod_{i=1}^\tau q_i\right)^{3/4+\eps}  \sharp\{ \bfk\in \Z_{\geq 0}^\tau: \prod_i q_i^{k_i}\leq N^2\} .
\end{equation*}
\end{lem}

\begin{proof}
Write $n = q_1^{l_1}\ldots q_\tau^{l_\tau}$. Without loss of generality, we may assume that $q_1,\ldots, q_h$ divide $n$ to some odd power and that $q_{h+1},\ldots, q_\tau$ divide $n$ to some even power. We first split the counting function $r(n)$ into different contributions according to what property of $B$ makes this value admissible for $n$. Hence, let $I_i$ for $1\leq i\leq 4$ be disjoint index sets with $I_1\cup I_2= \{1,\ldots, h\}$ and $I_3\cup I_4= \{h+1,\ldots, \tau\}$. Now let $r_I(n)$ be the number of integers $B$ which satisfy the following properties:\vspace{0.1cm}\\
i) $|B|\leq N$ and $|B+n|\leq N$,\vspace{0.2cm}\\
ii) $(B\mmod T)=b$,\vspace{0.1cm}\\
iii) for $i\in I_1$, one has $B=u_iq_i^{k_i}$ for some $(\frac{u_i}{q_i})=1$ and some even $k_i\geq 0$,\vspace{0.1cm}\\
iv) if $i\in I_2$, then $B=-n+u_iq_i^{k_i}$ for some even $k_i>l_i$ and $(\frac{u_i}{q_i})=1$,\vspace{0.1cm}\\
v) for $i\in I_3$, one has $B=u_iq_i^{k_i}$ for some even $k_i<l_i$ and $(\frac{u_i}{q_i})=1$,\vspace{0.1cm}\\
vi) $q_i^{l_i}|B$ for $i\in I_4$.\vspace{0.3cm}\\
By the definition of admissibility for $n$, we have
\begin{equation}\label{eqn4.1}
r(n)= \!\!\!\!\sum_{\substack{I_1\cup I_2=\{1,\ldots, h\}\\ I_1\cap I_2=\emptyset}} \sum_{\substack{I_3\cup I_4=\{h+1,\ldots, \tau\}\\ I_3\cap I_4=\emptyset}} \!\! r_I(n).
\end{equation}
{\em First step:} We evaluate each of the summands $\smash{r_I(n)}$ separately. Let $\smash{\bfk\in \Z_{\geq 0}^{\tau -|I_4|}}$ and define $\smash{r_I(n,\bfk)}$ to be the same counting function as $r_I(n)$ where we postulate properties iii)-v) with the exponent occurring exactly equal to the given $k_i$. Note that $r_I(n;\bfk)=0$ unless all the $k_i$ are even and $k_i>l_i$ for $i\in I_2$ and $k_i<l_i$ for $i\in I_3$. Furthermore, one has $\smash{r_I(n,\bfk)=0}$ if $\smash{\prod_{i\in I_1\cup I_3} q_i^{k_i}>N}$ or $\smash{\prod_{i\in I_2} q_i^{k_i}>N}$. Hence, we may rewrite $r_I(n)$ as
\begin{equation}\label{eqn4.2}
r_I(n)= \sum_{\bfk\in \Z_{\geq 0}^{\tau -|I_4|}}\!\! r_I(n,\bfk),
\end{equation}
which is a finite sum. Now we approximate $r_I(n,\bfk)$ for fixed even $\bfk$. For this, we parametrise the integers $B$ counted by $r_I(n,\bfk)$ in the following way. By conditions iii), iv) and vi), we have 
\begin{equation*}
B=t \!\!\prod_{i\in I_1\cup I_3} \!\!\!\! q_i^{k_i} \prod_{i\in I_4}q_i^{l_i}
\end{equation*}
for some $t$ such that $\smash{(t,\prod_{i\in I_1\cup I_3} q_i)=1}$. Furthermore, by iv), we have $B+n\equiv 0$ mod $\smash{\prod_{i\in I_2}q_i^{k_i}}$ and, since $k_i>l_i$ for $i\in I_2$, we obtain $\smash{t=t'\prod_{i\in I_2} q_i^{l_i}}$ for some integer $t'$. Again, by v), this integer $t'$ has to satisfy the congruence
$$ t'\!\!\prod_{i\in I_1\cup I_3} \!\!\!\! q_i^{k_i} \prod_{i\in I_4} q_i^{l_i} + (\prod_{i\in I_2} q_i^{-l_i})n\equiv 0 \hspace{0.2cm} \mmod \prod_{i\in I_2} q_i^{k_i-l_i}.$$
Since all the $q_i$ are distinct primes and the index sets $I_i$ are disjoint, this congruence has a unique solution $t_0$ for $t'$ modulo $\prod_{i\in I_2} q_i^{k_i-l_i}$. Hence, we may put 
$$ t'=t_0+u \prod_{i\in I_2} q_i^{k_i-l_i}.$$
Set $\varpi:= \prod_{i\in I_1\cup I_2\cup I_3} q_i^{k_i} \prod_{i\in I_4} q_i^{l_i}$, as well as $\varpi_2:= \prod_{i\in I_2} q_i^{k_i-l_i}$ and $\varpi_1:= \prod_{i\in I_1\cup I_3} q_i^{k_i}\prod_{i\in I_4}q_i^{l_i}$. Then we obtain 
$$B= (t_0+ u \varpi_2)\varpi_1 \prod_{i\in I_2} q_i^{l_i}.$$
Next we define $\nu_0\in\Z$ by $t_0\varpi_1+(\prod_{i\in I_2} q_i^{-l_i})n= \nu_0 \varpi_2$. Then we may rewrite $B+n$ as
$$B+n= (\nu_0+u\varpi_1) \prod_{i\in I_2} q_i^{k_i}.$$
The condition that $B$ runs through an interval given by $|B|\leq N$ and $|B+n|\leq N$ restricts the range of the new variable $u$ again to some bounded interval, which we call $\calJ$. Set $\smash{\mu:=t_0 \varpi_1 \prod_{i\in I_2} q_i^{l_i}}$ and $\smash{\nu:=\nu_0\prod_{i\in I_2} q_i^{k_i}}$. Then the function $r_I(n,\bfk)$ counts the number of integers $u\in \calJ$ with the following properties:\vspace{0.2cm}\\
a) the coprimality conditions $(t_0+u\varpi_2,\prod_{i\in I_1\cup I_3}q_i)=1$ and $(\nu_0+u\varpi_1,\prod_{i\in I_2}q_i)=1$~hold,\vspace{0.1cm}\\
b) $ (u\varpi + \mu \mmod T)= b$,\vspace{0.1cm}\\
c) for $i\in I_1\cup I_3$, one has $\smash{(\frac{q_i^{-k_i} (u\varpi+\mu)}{q_i})=1}$, and\vspace{0.1cm}\\
d) for $i\in I_2$, one has $(\frac{q_i^{-k_i}(u\varpi + \nu)}{q_i})=1$.\\\\
If $\bfk$ has even coordinates and $k_i>l_i$ for $i\in I_2$ and $k_i<l_i$ for $i\in I_3$, then we can now write $r_I(n,\bfk)$ in the form
\begin{equation*}
r_I(n,\bfk)= 2^{-\tau+|I_4|} \!\!\!\sum_{\substack{u\in \calJ \\  \mbox{ \scriptsize a), b) hold}}} \!\! \prod_{i\in I_1\cup I_3} \left(\left(\frac{q_i^{-k_i}(u\varpi+\mu)}{q_i}\right)+1\right)\prod_{i\in I_2} \left(\left(\frac{q_i^{-k_i}(u\varpi + \nu)}{q_i}\right)+1\right).\vspace{0.2cm}
\end{equation*}
{\em Second step:} Next, we remove the coprimality condition a). For this, write $\bfd=(d_1,d_2)$ and define
\begin{equation*}
r_I(n,\bfk,\bfd):= 2^{-\tau+|I_4|} \!\!\sum_{\substack{u\in \calJ, \mbox{\scriptsize b) holds}\\ d_1|t_0+u\varpi_2\\ d_2|\nu_0+u\varpi_1}} \prod_{i\in I_1\cup I_3} \left(\left(\frac{q_i^{-k_i}(u\varpi+\mu)}{q_i}\right)+1\right)\prod_{i\in I_2} \left(\left(\frac{q_i^{-k_i}(u\varpi + \nu)}{q_i}\right)+1\right).
\end{equation*}
Then we have
\begin{equation}\label{eqn4.3}
r_I(n,\bfk)= \!\!\sum_{d_1|\prod_{i\in I_1\cup I_3} q_i}\!\!\!\!\mu(d_1)\!\! \sum_{d_2|\prod_{i\in I_2}q_i}\!\!\mu(d_2) r_I(n,\bfk,\bfd).
\end{equation}
Let us consider one of the summands $r_I(n,\bfk,\bfd)$. Observe that, since $(d_1,\varpi_2)=1$, the congruence condition $t_0+u\varpi_2\equiv 0 \mmod d_1$ forces $u$ to lie in a unique congruence class modulo $d_1$, and similarly for the congruence $\nu_0+u\varpi_1\equiv 0 \mmod d_2$. Furthermore, the congruence condition b) 
$$(u\varpi +\mu \mmod T)= b $$
forces $u$ to lie in a fixed congruence class modulo $T$, since $(\varpi,T)=1$. Since all of the $d_1,d_2,T$ are coprime, we may substitute $u=u_0+d_1d_2Tx$ for some $u_0\in \Z$. The restriction $u\in \calJ$ is equivalent to $x\in \calJ'$ for some interval $\calJ'$, which we define by this property. Put
\begin{equation*}
\begin{split}
\alp:= u_0 \prod_{i\in I_2} \!\! q_i^{k_i}\prod_{i\in I_4}\!\! q_i^{l_i}+ t_0 \!\!\!\prod_{i\in I_2\cup I_4}\!\! q_i^{l_i},\ \quad\quad\bet:= d_1d_2T\prod_{i\in I_2} \!\! q_i^{k_i}\prod_{i\in I_4}\!\! q_i^{l_i},\\  \gam:= u_0 \!\!\prod_{i\in I_1\cup I_3}\!\! \!\! q_i^{k_i}\prod_{i\in I_4} \!\! q_i^{l_i}+\nu_0,\ \quad\quad\del := d_1d_2T\!\! \prod_{i\in I_1\cup I_3}\!\!\!\! q_i^{k_i}\prod_{i\in I_4}\!\! q_i^{l_i},
\end{split}
\end{equation*}
such that we have
$$ u\varpi +\mu= (\alp +\bet x) \!\!\prod_{i\in I_1\cup I_3}\! \! q_i^{k_i} \quad \mbox{ and } \quad u\varpi +\nu= (\gam+\del x) \prod_{i\in I_2} q_i^{k_i}.$$
Set $\varpi_3= \prod_{i\in I_1\cup I_3} q_i^{k_i}$ and $\varpi_4=\prod_{i\in I_2}q_i^{k_i}$. Then we have
\begin{equation}\label{eqn4.4}
r_I(n,\bfk,\bfd)= 2^{-\tau +|I_4|}\sum_{x\in \calJ'} \prod_{i\in I_1\cup I_3} \left(\left(\frac{q_i^{-k_i}\varpi_3(\alp+\bet x)}{q_i}\right)+1\right)\prod_{i\in I_2} \left(\left(\frac{q_i^{-k_i}\varpi_4(\gam+\del x)}{q_i}\right)+1\right).\vspace{0.2cm}
\end{equation}
{\em Third step:} Let $I_i'$ for $1\leq i\leq 3$ be subsets of $I_i$, and consider the sum
\begin{equation*}
\sum_{x\in \calJ'}\prod_{i\in I_1'\cup I_3'} \left(\frac{q_i^{-k_i}\varpi_3(\alp+\bet x)}{q_i}\right)\prod_{i\in I_2'}\left(\frac{q_i^{-k_i}\varpi_4(\gam+\del x)}{q_i}\right).
\end{equation*}
We aim to give an upper bound for this character sum. By the definitions of $\varpi_3$ and $\varpi_4$, and from the fact that all $k_i$ are even, we see that it coincides with the sum
\begin{equation*}
E_{I'}:= \sum_{x\in \calJ'}\prod_{i\in I_1'\cup I_3'} \left(\frac{\alp+\bet x}{q_i}\right)\prod_{i\in I_2'}\left(\frac{\gam+\del x}{q_i}\right).
\end{equation*}
Next, let us choose a complete set of residues modulo $\smash{\prod_{i\in I_1'\cup I_3'}q_i}$, which we call $\calT\subset \Z$, with the property that $\gam + \del y \equiv 0$ modulo $\smash{\prod_{i\in I_2}q_i}$ for all $y\in \calT$. For this, we need to make sure that if $q_i$, for $i\in I_2$, divides $\del$, then it also divides $\gam$. This is the case by definition of $\smash{\del= d_1d_2T\prod_{i\in I_1\cup I_3}q_i^{k_i}\prod_{i\in I_4}q_i^{l_i}}$ and $\gam= u_0\varpi_1+\nu_0\equiv 0$ modulo $d_2$. Indeed, we have constructed $u_0$ in such a way that $d_2|u_0\varpi_1+\nu_0$. We now sort the elements $x\in \calJ'$ into these residue classes modulo $\smash{\prod_{i\in I_1'\cup I_3'}q_i}$, and write $\smash{x=y+z\prod_{i\in I_1'\cup I_3'} q_i}$ for $y\in \calT$ and $x\equiv y $ modulo $\smash{\prod_{i\in I_1'\cup I_3'}q_i}$. For each fixed $y\in \calT$, there is some interval $\calJ''(y)$ such that for all $x$ in this residue class $y$ one has $x\in \calJ'$ if and only if $z\in \calJ''(y)$. We rewrite $E_{I'}$ as
\begin{equation*}
E_{I'}= \sum_{y\in \calT} \prod_{i\in I_1'\cup I_3'} \left(\frac{\alp+\bet y}{q_i}\right)\sum_{z\in \calJ''(y)} \prod_{i\in I_2'} \left(\frac{z\del \prod_{j\in I_1'\cup I_3'}q_j }{q_i}\right).
\end{equation*}
If there is some $i\in I_2'$ with $q_i|\del$, then $E_{I'}=0$ trivially. Otherwise, we use the Polya-Vinogradov inequality (see equation (51), p. 263 in \cite{Tenenbaum}) for multiplicative characters to deduce the bound
\begin{equation*}
E_{I'}\ll \!\!\prod_{i\in I_1'\cup I_3'}\!\!\! q_i \left(\prod_{i\in I_2'}q_i\right)^{1/2} \!\!\log (\prod_{i=1}^\tau q_i).
\end{equation*}
In reversing the roles of $\prod_{i\in I_1'\cup I_3'} q_i$ and $\prod_{i\in I_2'} q_i$, we obtain a similar bound with these two terms interchanged, and hence conclude that
\begin{equation*}
E_{I'}\ll (\!\!\prod_{i\in I_1'\cup I_2'\cup I_3'}\!\! \!\!q_i)^{3/4}  \log (\prod_{i=1}^\tau q_i).\vspace{0.2cm}
\end{equation*}
{\em Fourth step:} Using this bound, we may now rewrite the function $r_I(n,\bfk,\bfd)$ in (\ref{eqn4.4}) as
\begin{equation*}
\begin{split}
r_I(n,\bfk,\bfd)&=\sum_{x\in \calJ'}2^{-\tau +|I_4|}+O\left( 2^{-\tau+|I_4|}2^{|I_1|+|I_2|+|I_3|} (\!\!\prod_{i\in I_1'\cup I_2'\cup I_3'}\!\!\!\! q_i)^{3/4}  \log (\prod_{i=1}^\tau q_i)\right)\\
&=2^{-\tau+|I_4|}(|\calJ'|+O(1))+ O\left( (\!\!\prod_{i\in I_1'\cup I_2'\cup I_3'}\!\! \!\! q_i)^{3/4}  \log (\prod_{i=1}^\tau q_i)\right).
\end{split}
\end{equation*}
We compute the length of the interval $\calJ'$ as 
$$ |\calJ'|= (Td_1d_2\varpi)^{-1}(2N-|n|+1)$$
and deduce that
\begin{equation*}
r_I(n,\bfk,\bfd)= 2^{-\tau+|I_4|}(Td_1d_2\varpi)^{-1}(2N-|n|+1)+O\left( (\!\!\prod_{i\in I_1'\cup I_2'\cup I_3'}\!\!\!\!\! q_i)^{3/4}  \log (\prod_{i=1}^\tau q_i)\right).
\end{equation*}
By equation (\ref{eqn4.3}), we obtain
\begin{equation*}
r_I(n,\bfk)= 2^{-\tau+|I_4|} \frac{2N-|n|+1}{T\varpi} \!\!\!\! \sum_{d|\prod_{i\in I_1\cup I_2\cup I_3}q_i}\!\!\!\!\!\! \frac{\mu(d)}{d} +O \left( (\!\!\prod_{i\in I_1'\cup I_2'\cup I_3'}\!\! \!\! q_i)^{3/4}  \log (\prod_{i=1}^\tau q_i)\right).
\end{equation*}
Let $\calK$ be the set of vectors $\bfk\in \Z_{\geq 0}^{\tau -|I_4|}$ such that all coordinates $k_i$ are even and $k_i>l_i$ for $i\in I_2$ and $k_i<l_i$ for $i\in I_3$. Furthermore, let $\calK(N)$ be the intersection of $\calK$ with the set of tuples $\smash{\bfk\in \Z_{\geq 0}^{\tau -|I_4|}}$ such that $\smash{\prod_{i\in I_1\cup I_3} q_i^{k_i}\leq N}$ and $\smash{\prod_{i\in I_2} q_i^{k_i}\leq N}$. Then we obtain by equation (\ref{eqn4.2})
\begin{equation}\label{eqn4.5}
r_I(n)= \!\!\sum_{\bfk\in \calK(N)} \!\!2^{-\tau+|I_4|} \frac{2N-|n|+1}{T\varpi} \sum_{d|\prod_{i\in I_1\cup I_2\cup I_3}q_i}\!\!\!\! \frac{\mu(d)}{d}+O(E_4),
\end{equation}
with an error term $E_4$ bounded by
\begin{equation*}
E_4\ll (\prod_{i=1}^\tau q_i)^{3/4+\eps} \sharp\{\bfk\in \Z_{\geq 0}^\tau: \prod_{i=1}^\tau q_i^{k_i}\leq N^2\}.\vspace{0.2cm}
\end{equation*}
{\em Fifth step:} We next complete the sum in (\ref{eqn4.5}) over all $\bfk\in \calK$. Note that it is absolutely convergent, and more precisely one has
\begin{equation*}
\sum_{\bfk\in \calK\setminus \calK(N)}\prod_{i\in I_1\cup I_2\cup I_3}\!\!\! q_i^{-k_i}\ll \tau 2^\tau N^{-1} \sharp\{\bfk\in \Z_{\geq 0}^\tau: \prod_{i=1}^\tau q_i^{k_i}\leq N^2\}.
\end{equation*}
Hence, we obtain
\begin{equation*}
r_I(n)= \sum_{\bfk\in \calK}  2^{-\tau+|I_4|} \frac{2N-|n|+1}{T\varpi} \prod_{i\in I_1\cup I_2\cup I_3}\!\!\! \left(1-\frac{1}{q_i}\right)+O(E_4).
\end{equation*}
We finally come back to equation (\ref{eqn4.1}) to evaluate $r(n)$ as
\begin{equation}\label{eqn4.6}
r(n)= \frac{2N-|n|+1}{T}\!\!\sum_{\substack{I_1\cup I_2=\{1,\ldots, h\}\\ I_1\cap I_2=\emptyset}} \sum_{\substack{I_3\cup I_4=\{h+1,\ldots, \tau\}\\ I_3\cap I_4=\emptyset}} \!\!\rho(I)+O(2^\tau E_4),
\end{equation}
with
\begin{equation*}
\rho(I):=\sum_{\bfk\in \calK}\prod_{i\in I_4}q_i^{-l_i}\!\! \!\!\prod_{i\in I_1\cup I_2\cup I_3}\!\!\left(q_i^{-k_i}2^{-1}(1-\frac{1}{q_i})\right) .
\end{equation*}
We compute $\rho(I)$ as
\begin{equation*}
\begin{split}
\rho(I)= &\prod_{i\in I_1} \left(2^{-1}(1+q_i^{-1})^{-1}\right) \prod_{i\in I_2} \left(2^{-1}q_i^{-(l_i+1)}(1+q_i^{-1})^{-1}\right)\\ &\hspace{1.5cm} \cdot \prod_{i\in I_3} \left( 2^{-1}(1+q_i^{-1})^{-1} (1-q_i^{-l_i})\right) \prod_{i\in I_4} q_i^{-l_i}.
\end{split}
\end{equation*}
By the definition of the multiplicative function $\sig(n)$, we conclude that
\begin{equation*}
r(n)= \frac{2N-|n|+1}{T} \sig(n) +O(2^\tau E_4),
\end{equation*}
which establishes the lemma.

\end{proof}

Before we treat the main term arising from Lemma \ref{lem4.4} in the asymptotic for $\RBR (N;T,a,b)$, let us show that the contribution of the error term $r_1(n)$ in Lemma \ref{lem4.4} is negligible.

\begin{lem}\label{lem4.5}
Let $r_1(n)$ be as in Lemma \ref{lem4.4}. Then one has
\begin{equation*}
\sum_{1\leq n\leq N} \!\!\! r_1(n) \ll_\eps  N^{7/4+\eps},
\end{equation*}
with an implied constant independent of $T$.
\end{lem}

\begin{proof}
It is sufficient to show that
\begin{equation}\label{eqn4.30}
R_\tau(N;\bfq):=\sharp\{ \bfk\in \Z_{\geq 0}^\tau: \prod_{i=1}^\tau q_i^{k_i}\leq N^2\} \ll_\eps N^{\eps},
\end{equation}
for any $n\leq N$. Hence, we assume that $q_1<\ldots <q_\tau$ are primes with $\prod_{i=1}^\tau q_i |n$. Since we are only interested in upper bounds, we may even assume that $q_1<\ldots <q_\tau$ are the first $\tau$ primes. Note that
\begin{equation*}
R_\tau(N;\bfq)= \sharp\{ \bfk\in \Z_{\geq 0}^\tau: \sum_{i=1}^\tau k_i \log q_i \leq 2\log N\}.
\end{equation*}
We claim that
\begin{equation*}
R_\tau(N;\bfq)\leq \vol \{ \bfk\in \R_{\geq 0}^\tau: \sum_{i=1}^\tau  k_i \log q_i \leq 3\log N\}.
\end{equation*}
This holds since, for any $\bfk$ counted by $R_\tau(N;\bfq)$, one has $\prod_{i=1}^\tau q_i\ \leq N$ and hence
\begin{equation*}
\sum_{i=1}^\tau (k_i+1)\log q_i \leq 2\log N+ \sum_{i=1}^\tau \log q_i \leq 3 \log N.
\end{equation*}
Next, we observe the volume of the simplex arising is
\begin{equation*}
\vol \{\bfk\in \R_{\geq 0}^\tau: \sum_{i=1}^\tau k_i \log q_i \leq 3\log N\} = \frac{1}{\tau!} \frac{(3\log N)^\tau}{\prod_{i=1}^\tau \log q_i}.
\end{equation*}
We need to get an upper bound for the last expression. For this, we first need a rough upper bound for $\tau$. Note that there is some positive constant $C_1$, such that
\begin{equation*}
C_1q_\tau\leq \sum_{q\leq q_\tau} \log q \leq \log N,
\end{equation*}
where the summation is over all prime numbers $q$. Here we used that $\prod_{i=1}^\tau q_i\leq N$. Now we obtain by the prime number theorem
\begin{equation*}
\tau\leq \sharp\{ q\leq q_\tau: q \mbox{ prime }\} \leq C_2 \frac{q_\tau}{\log q_\tau} \leq C_3 \frac{\log N}{\log \log N},
\end{equation*}
for some positive constants $C_2$ and $C_3$.\par
We are now in a position to estimate the size of
\begin{equation*}
\begin{split}
\log \left[ \frac{1}{\tau!} \frac{(3\log N)^\tau}{\prod_{i=1}^\tau \log q_i}\right] &= \tau \log 3 + \tau \log \log N - \sum_{i=1}^\tau \log i - \sum_{i=1}^\tau \log \log q_i\\
&= \tau \log 3 + \tau\log \log N-\tau \log \tau +\tau -\sum_{i=1}^\tau \log \log q_i +O(\log \tau)\\
&= \tau \log \log N+\tau-\tau \log \tau+ O\left(\frac{\log N}{\log \log N} \log \log \log N\right).
\end{split}
\end{equation*}
The derivative of the function $g(\tau):= -\tau \log \tau +\tau +\tau \log \log N$ is given by $\log \log N-\log \tau$, and hence $g(\tau)$ is increasing for $\tau< \log N$. For $N$ sufficiently large, we may therefore apply the bound $\tau\leq C_3 \frac{\log N}{\log \log N}$, and obtain
\begin{equation*}
\begin{split}
\log \left[ \frac{1}{\tau!} \frac{(2\log N)^\tau}{\prod_{i=1}^\tau \log q_i}\right] &\leq C_3 \frac{\log N}{\log\log N} \log\log N - C_3 \left(\frac{\log N}{\log\log N} \right)\log \left(C_3 \frac{\log N}{\log\log N}\right) \\&\hspace{4cm}+O\left( \frac{\log N}{\log\log N} \log\log\log N\right)\\ &=C_3 \frac{\log N}{\log\log N}\left(\log\log\log N-\log C_3\right)+O\left( \frac{\log N}{\log\log N} \log\log\log N\right)\\ &= O\left( \frac{\log N}{\log\log N} \log\log\log N\right).
\end{split}
\end{equation*}
This establishes the bound (\ref{eqn4.30}) with an implied constant depending on $\eps$.
\end{proof}

Next, we aim to evaluate the sum
\begin{equation}\label{eqn4.31}
\Sigma_1:=\!\!\sum_{\substack{1\leq n\leq 2N\\ n\equiv b-a \mmod T}}\!\!\!\sig(n).
\end{equation}
For this, we let $T'=\gcd (b-a,T)$ and $T''=T/T'$. Then we may rewrite the sum $\Sigma_1$ as
\begin{equation*}
\Sigma_1= \!\!\sum_{\substack{1\leq T'm\leq 2N\\ T'm\equiv b-a \mmod T}} \!\!\!\!\!\!\!\sig(T' m).
\end{equation*}
Let $b-a=T'd$ for some $d$ modulo $T''$, and further rewrite $\Sigma_1$ as
\begin{equation*}
\Sigma_1= \!\!\sum_{\substack{1\leq m\leq N'\\ m\equiv d \mmod T'' }} \!\!\!\!\!\!\sig(m),
\end{equation*}
with $N'=\frac{2N}{T'}$. We encode the condition $m\equiv d \mod T''$ using multiplicative characters modulo~$T''$, and obtain
\begin{equation}\label{eqnSigma}
\Sigma_1= \frac{1}{\varphi(T'')}\sum_{\chi \mmod T''}\!\! \!\!\overline{\chi}(d) \sum_{1\leq m\leq N'} \!\!\!\!\chi(m) \sig(m) = \frac{1}{\varphi(T'')}\sum_{\chi \mmod T''} \!\!\!\!\overline{\chi}(d) \Sigma_1(\chi),
\end{equation}
with sums of the form
\begin{equation*}
\Sigma_1(\chi)= \sum_{1\leq m\leq N'}\!\! \!\!\chi(m) \sig(m),
\end{equation*}
for any multiplicative character $\chi$ modulo $T''$. These can be evaluated via an application of Perron's formula. For this, let $D_\chi(s)$ be the associated Dirichlet series, given by
\begin{equation*}
D_\chi(s):=\sum_{m=1}^\infty \frac{\chi(m)\sig(m)}{m^s}.
\end{equation*}
It is clear that $D_\chi(s)$ is absolutely convergent for $\Re(s)>1$. In this region, it can be expressed as an Euler product
\begin{equation*}
D_\chi(s)= \prod_{p|D} \left(1-\frac{\chi(p)}{p^s}\right)^{-1}\prod_{\left(\frac{D}{p}\right)=1}\!\!\left(1-\frac{\chi(p)}{p^s}\right)^{-1}\prod_{\left(\frac{D}{p}\right)=-1}\!\!\left(1+\sum_{l=1}^\infty\frac{\sig(p^l)\chi(p^l)}{p^{ls}}\right).
\end{equation*}

We next compare the Dirichlet series $D_\chi(s)$ to products of Dirichlet $L$-functions. For some character $\chi$ modulo $T''$, we write
\begin{equation*}
L(s,\chi)=\sum_{n=1}^\infty \frac{\chi(n)}{n^s}.
\end{equation*}

\begin{lem}\label{lem4}
One has
\begin{equation*}
D_\chi(s)^4= L(s,\chi)^3 L\left(s,\left(\frac{\cdot}{D}\right)\chi\right) H^{(3)}(s),
\end{equation*}
where $H^{(3)}(s)$ is given by some Euler product in $\Re(s)>1/2$, which is absolutely convergent in this region.
\end{lem}

\begin{proof}
We rewrite the Euler product of $D_\chi(s)$ as 
\begin{equation*}
D_\chi(s)= \prod_{\left(\frac{D}{p}\right)=1}\!\!\left(1-\frac{\chi(p)}{p^s}\right)^{-1} \!\!\! \prod_{\left(\frac{D}{p}\right)=-1}\!\!\left(1-\frac{\chi(p)}{2p^s}\right)^{-1} \cdot H^{(1)}(s),
\end{equation*}
for 
$$H^{(1)}(s):=\prod_{p|D} \left(1-\frac{\chi(p)}{p^s}\right)^{-1}\!\!\!\!\prod_{\left(\frac{D}{p}\right)=-1}\!\! \!\!\!g_p(s),$$
and
\begin{equation*}
g_p(s):= \left(1-\frac{\chi(p)}{2p^s}\right)\left(1+\sum_{l=1}^\infty\frac{\sig(p^l)\chi(p^l)}{p^{ls}}\right).
\end{equation*}
Note that $H^{(1)}(s)$ is absolutely convergent in $\Re(s)>1/2$. 
We compute the product
\begin{equation*}
\begin{split}
L(s,\chi)L\left(s,\left(\frac{\cdot}{D}\right)\chi\right) 
= \prod_{p|D} \left(1-\frac{\chi(p)}{p^s}\right)^{-1} \!\!\!\prod_{\left(\frac{p}{D}\right)=1}\!\!\left(1-\frac{\chi(p)}{p^s}\right)^{-2} \!\!\!\prod_{\left(\frac{p}{D}\right)=-1}\!\!\left(1-\frac{\chi(p^2)}{p^{2s}}\right)^{-1}.
\end{split}
\end{equation*}
Note that $\smash{(\frac{D}{p})=(\frac{p}{D})}$ since $D\equiv 1$ mod $4$. On the other hand, we consider the square of the Dirichlet function $D_\chi(s)$, which is given by
\begin{equation*}
\begin{split}
D_\chi(s)^2&= \prod_{\left(\frac{D}{p}\right)=1}\!\!\left(1-\frac{\chi(p)}{p^s}\right)^{-2} \!\!\!\!\prod_{\left(\frac{D}{p}\right)=-1}\!\!\left(1-\frac{\chi(p)}{p^s}+\frac{\chi(p^2)}{4p^{2s}}\right)^{-1} \cdot H^{(1)}(s)^2\\
&= \prod_{\left(\frac{D}{p}\right)=1}\!\!\left(1-\frac{\chi(p)}{p^s}\right)^{-2} \!\!\!\!\prod_{\left(\frac{D}{p}\right)=-1}\!\!\left(1-\frac{\chi(p)}{p^s}\right)^{-1} \cdot H^{(2)}(s),
\end{split}
\end{equation*}
for
\begin{equation*}
H^{(2)}(s):= H^{(1)}(s)^2\prod_{\left(\frac{p}{D}\right)=-1}\left[\left(1-\frac{\chi(p)}{p^s}\right)\left(1-\frac{\chi(p)}{p^s}+\frac{\chi(p^2)}{4p^{2s}}\right)^{-1}\right].
\end{equation*}
Hence, we obtain
\begin{equation*}
\begin{split}
D_\chi(s)^4&= \prod_{\left(\frac{D}{p}\right)=1}\!\!\left(1-\frac{\chi(p)}{p^s}\right)^{-4} \!\!\!\!\prod_{\left(\frac{D}{p}\right)=-1}\!\!\left(1-\frac{\chi(p)}{p^s}\right)^{-2} \cdot H^{(2)}(s)^2\\
&=L(s,\chi)^2 H^{(2)}(s)^2\prod_{\left(\frac{D}{p}\right)=1}\!\!\left(1-\frac{\chi(p)}{p^s}\right)^{-2} \prod_{p|D}\left(1-\frac{\chi(p)}{p^s}\right)^2\\
&=L(s,\chi)^2 H^{(2)}(s)^2L(s,\chi)L\left(s,\left(\frac{\cdot}{D}\right)\chi\right)\prod_{\left(\frac{D}{p}\right)=-1}\!\!\left(1-\frac{\chi(p^2)}{p^{2s}}\right)\prod_{p|D}\left(1-\frac{\chi(p)}{p^s}\right)^3.
\end{split}
\end{equation*}
Let $H^{(3)}(s)$ be given by
\begin{equation*}
H^{(3)}(s):=H^{(2)}(s)^2\prod_{\left(\frac{p}{D}\right)=-1}\!\!\left(1-\frac{\chi(p^2)}{p^{2s}}\right)\prod_{p|D}\left(1-\frac{\chi(p)}{p^s}\right)^3,
\end{equation*}
and note that $H^{(3)}(s)$ is absolutely convergent in $\Re(s)>1/2$. We summarize our calculation above as 
\begin{equation*}
D_\chi(s)^4= L(s,\chi)^3 L\left(s,\left(\frac{\cdot}{D}\right)\chi\right) H^{(3)}(s),
\end{equation*}
 which completes the proof of the lemma
\end{proof}

Next, we evaluate the sum $\Sigma_1(\chi)$ asymptotically for the trivial character $\chi=\chi_0$, and show that the contribution from all non-trivial characters is negligible or corresponds to lower order terms.

\begin{lem}\label{lem5}
a) One has
\begin{equation*}
\Sigma_1(\chi_0)=\frac{N'}{(\log N')^{1/4}} \sum_{k=0}^P \frac{\lam_k}{(\log N')^{k}} + O_{D,P} \left(\frac{N'}{(\log N')^{1/4+P+1}}\right),
\end{equation*}
for some real constants $\lam_k$, $0\leq k\leq P$. More precisely, one has $\lam_0= \frac{G(1)}{\Gam\left(\frac{3}{4}\right)}$, for 
\begin{equation*}
G(1):=\prod_{p|T''}\!\!\left(1-p^{-1}\right)^{3/4}\!\!\!\! \prod_{p|\frac{D}{(D,T'')}}\!\!\!\!\!\left(1-p^{-1}\right)^{-1/4} L\left(1,\left(\frac{\cdot}{D}\right)\chi_0\right)^{1/4} \!\!\!\!\prod_{\left(\frac{D}{p}\right)=-1} \!\!\!\!\! c_p
\end{equation*}
and
\begin{equation*}
c_p= (1-p^{-1})^{3/4} (1+p^{-1})^{1/4} \left(1+\sum_{l=1}^\infty \frac{\sig(p^l)}{p^l}\right).
\end{equation*}
Furthermore, the product defining $G(1)$ is absolutely convergent and $G(1)>0$. The constants $\lam_k$ are given by $\lam_k=\lam_k\left(\frac{3}{4}\right)$, as defined in equation (15) in \S II.5 of \cite{Tenenbaum}.\vspace{0.1cm}\\
b) Let $A_0>0$ be some real parameter and assume that $T''\leq (\log N')^{A_0}$. Then there is a (ineffective) constant $C(A_0)$ with the following property. If $\chi\neq \chi_0$ and $\chi \left(\frac{\cdot}{D}\right)$ is a non-trivial character, then one has the bound
\begin{equation*}
\Sigma_1(\chi)\ll N' e^{-C(A_0)\sqrt{\log N'}}.\vspace{0.1cm}
\end{equation*}
c) If $\chi$ is a non-trivial character modulo $T''$ such that $\chi\left(\frac{\cdot}{D}\right)$ is the trivial character modulo $D$, then one has
\begin{equation*}
\Sigma_1(\chi)= \frac{N'}{(\log N')^{3/4}} \sum_{k=0}^P \frac{\mu_k}{(\log N')^{k}} +O_{D,P} \left(\frac{N'}{(\log N')^{3/4+P+1}}\right),
\end{equation*}
for some real numbers $\mu_k$. 
\end{lem}


Having established Lemma \ref{lem4}, we are already prepared to use the Selberg-Delange method to evaluate $\Sigma_1(\chi)$.

\begin{proof}
First, we prove a), i.e.~treat the case $\chi=\chi_0$. Note that, for $D$ fixed, there is only a finite number of trivial characters modulo $T''$, where $T''$ varies over all moduli which are composed of primes dividing $D$. Hence, all our estimates for $\Sig_1(\chi_0)$ are uniform in $T''$ and the implicit constants depend only on~$D$.\par
By Lemma \ref{lem4}, we see that the function 
\begin{equation}\label{eqn19}
G(s):=D_{\chi_0}(s)\zet(s)^{-\frac{3}{4}}
\end{equation}
may be continued as a holomorphic function to the region $\sig\geq 1-c_0/(1+\log(3+|t|))$, where $s=\sig +it$.
Since $H^{(3)}(s)$ is given as an Euler product in $\Re(s)>1/2$, which is absolutely convergent in this region, we may apply Theorem 3 in \S II.5 in \cite{Tenenbaum}. We obtain for $N'\geq 3$ the asymptotic formula
\begin{equation*}
\Sigma_1(\chi_0)=\frac{N'}{(\log N')^{1/4}} \sum_{k=0}^P \frac{\lam_k}{(\log N')^{k}} + O_{D,P} \left(\frac{N'}{(\log N')^{1/4+P+1}}\right),
\end{equation*}
where $\smash{\lam_k=\lam_k\!\!\left(\frac{3}{4}\right)}$ is defined as in equation (15) in \S II.5 in \cite{Tenenbaum}. In particular, one has $\smash{\lam_0= \frac{G(1)}{\Gam\left(\frac{3}{4}\right)}}$. To find the constant $G(1)$, we recall that Lemma \ref{lem4}, together with the definition (\ref{eqn19}) of $G$, shows 
\begin{equation*}
G(s)= \prod_{p|T''} \!\left(1-p^{-s}\right)^{3/4} L\left(s,\left(\frac{\cdot}{D}\right)\chi_0\right)^{1/4} H^{(3)}(s)^{1/4}.
\end{equation*}
A short calculation reveals that
\begin{equation*}
G(1)= \prod_{p|T''}\!\left(1-p^{-1}\right)^{3/4} \!\!\!\!\prod_{p|\frac{D}{(D,T'')}}\!\!\!\!\left(1-p^{-1}\right)^{-1/4} L\left(1,\left(\frac{\cdot}{D}\right)\chi_0\right)^{1/4} \!\!\!\prod_{\left(\frac{D}{p}\right)=-1}\!\! \!\! c_p,
\end{equation*}
with constants $c_p$ given by
\begin{equation*}
c_p= (1-p^{-2})^{1/4}\left(1-\frac{1}{2p}\right) \left(1+\sum_{l=1}^\infty \frac{\sig(p^l)}{p^l}\right) \left( 1-\frac{1}{p}\right)^{1/2} \left(1-\frac{1}{p}+\frac{1}{4p^2}\right)^{-1/2}.
\end{equation*}
This can be simplified to
\begin{equation*}
c_p= (1-p^{-1})^{3/4} (1+p^{-1})^{1/4} \left(1+\sum_{l=1}^\infty \frac{\sig(p^l)}{p^l}\right).\vspace{0.1cm}
\end{equation*}
b) Similarly, one can use Lemma \ref{lem4} in combination with an application of Perron's formula to deduce the upper bounds on $\Sig_1(\chi)$, for $\chi\neq \chi_0$ and $\chi \left(\frac{\cdot}{D}\right)$ non-trivial. The computations are similar to the Siegel-Walfisz theorem (but simpler) and we omit the details here.\vspace{0.1cm}\\
c) The last part of the lemma follows in a fashion similar to the first part, via an application of the Selberg-Delange method as in \S II.5 in \cite{Tenenbaum}.
\end{proof}


Let $A_0>0$ be some real parameter and $T''\leq (\log N')^{A_0}$. 
From Lemma \ref{lem5}, we now conclude in combination with equation (\ref{eqnSigma}) that
\begin{equation}\label{eqnSigma1}
\Sig_1= \frac{2N}{\varphi(T'')T'\left(\log \frac{2N}{T'}\right)^{1/4}} \sum_{k=0}^{2P} \frac{\lamtil_k(T'')}{\left(\log \frac{2N}{T'}\right)^{k/2}} +O_{D,P}\left(\frac{N}{T'}\left(\frac{1}{\varphi(T'')\left(\log \frac{2N}{T'}\right)^{P+3/4}}+e^{-C(A_0)\sqrt{\log N'}}\right)\right),
\end{equation}
where the constants $\lamtil_k(T'')$ are defined via
\begin{equation*}
\lamtil_{2k}(T'')=\lam_k,\quad \quad \lamtil_{2k+1}(T'')= \mu_k,
\end{equation*}
and $\lam_k$ and $\mu_k$ are as in Lemma \ref{lem5}.\par
We furthermore define
\begin{equation*}
\Sigma_2:= \!\!\!\!\!\sum_{\substack{1\leq n\leq 2N\\ n\equiv b-a \mmod T}}\!\!\!\!\!\! n\sig(n).
\end{equation*}
We evaluate $\Sig_2$ using partial summation and our asymptotic for $\Sig_1$ in (\ref{eqnSigma1}). This leads to
\begin{equation}\label{eqnSigma2}
\Sig_2=\frac{2N^2}{\varphi(T'')T'\left(\log \frac{2N}{T'}\right)^{1/4}} \sum_{k=0}^{2P} \frac{\lamtil'_i(T'')}{\left(\log \frac{2N}{T'}\right)^{k/2}} +O_{D,P}\left(\frac{N^2}{T'}\left(\frac{1}{\varphi(T'')\left(\log \frac{2N}{T'}\right)^{P+3/4}}+e^{-C(A_0)\sqrt{\log N'}}\right)\right),
\end{equation}
with real constants $\lamtil'_i(T'')$ and $\lamtil'_0(T'')=\lam_0$. 

We are now in a position to collect our results of this section in the following theorem.

\begin{theo}\label{thm5}
Assume that $D$ is some positive integer with $D\equiv 1 \mmod 8$, which is square-free, and $T=\prod_{i=1}^r p_i^{l_i}$ be a modulus composed of primes dividing $D$. Let $a$ and $b$ be congruence classes modulo $T$, which satisfy $(a,b)\in \calH(\bfl)$. Finally, let $A_0>0$ be some real parameter and $\smash{T\leq (\log N)^{A_0}}$. Then there are real constants $c_i(T'')$ with $\smash{c_0(T'')=\frac{G(1)}{\Gam\left(\frac{3}{4}\right)}}$ such that one has
\begin{equation*}
\RBR (N;T,a,b)= \frac{4N^2}{T\phi(T'')T'\left(\log \frac{2N}{T'}\right)^{1/4}} \sum_{k=0}^{2P} \frac{c_k(T'')}{\left(\log \frac{2N}{T'}\right)^{k/2}} +O_{D,P,A_0}\left(\frac{N^2}{T^2 (\log N)^{P+3/4}}\right).
\end{equation*}
The constant $G(1)$ is given as in Lemma \ref{lem5}, and $T'$ and $T''$ are defined by $T'=\gcd (b-a,T)$ and $T''=T/T'$. The constant in the last error term is ineffective in $A_0$. Moreover, one has $c_k(T'')\ll_{D,P}1$, for all $0\leq k\leq 2P$. 
\end{theo}

\begin{proof}
We start with the relation from equation (\ref{eqn4.10}), which asserts that
\begin{equation*}
\RBR (N;T,a,b)= \!\!\sum_{\substack{n\equiv a-b \mmod T\\ |n|\leq 2N}}\!\!\!\!\!\! r(n)+O_\eps (D^{\eps} N^{1+\eps}).
\end{equation*}
We decompose $r(n)$ according to Lemma \ref{lem4.4}, and obtain
\begin{equation*}
\RBR (N;T,a,b)= \!\!\!\!\!\!\sum_{\substack{n\equiv a-b \mmod T\\ |n|\leq 2N}}\!\!\frac{2N-|n|+1}{T} \sig(n)+ \sum_{|n|\leq 2N}\!\!\! r_1(n) +O_\eps (D^{\eps} N^{1+\eps}).
\end{equation*}
Lemma \ref{lem4.5} implies that
\begin{equation*}
\RBR (N;T,a,b)= \!\!\sum_{\substack{n\equiv a-b \mmod T\\ |n|\leq 2N}}\!\!\!\!\!\!\frac{2N-|n|+1}{T} \sig(n)+O_{\eps,D}(  N^{7/4+\eps}).
\end{equation*}
We recall that $\sig(-1)=1$ and hence
\begin{equation*}
\sum_{\substack{n\equiv a-b \mmod T\\ -2N\leq n<0}}\!\!\! \sig(n)= \sum_{\substack{ n\equiv b-a\mmod T\\ 1\leq n\leq 2N}} \!\!\!\sig(n).
\end{equation*}
This is evaluated in the very same way as $\Sig_1$ (see equation (\ref{eqn4.31})). A combination of the asymptotics in (\ref{eqnSigma1}) and (\ref{eqnSigma2}) leads to
\begin{equation*}
\begin{split}
\RBR (N;T,a,b)
= \frac{4N^2}{T\phi(T'')T'\left(\log \frac{2N}{T'}\right)^{1/4}} \sum_{k=0}^{2P} \frac{c_k(T'')}{\left(\log \frac{2N}{T'}\right)^{k/2}} +O_{D,P,A_0}\left(\frac{N^2}{T^2 (\log N)^{P+3/4}}\right),
\end{split}
\end{equation*}
with real constants $c_k(T'')$ and $c_0(T'')= \frac{G(1)}{\Gam\left(\frac{3}{4}\right)}$. 
This completes the proof of the theorem.
\end{proof}

\section{Proof of the main theorem}
Let $\RBR (N)$ be the number of del Pezzo surfaces $S^{(D;A,B)}$ of degree four in the family (\ref{eqn0}) of height at most $H(S^{(D;A,B)})\leq N$ that are counterexamples to the Hasse principle explained by some Brauer-Manin obstruction. In order to compute $\RBR (N)$, we argue similarly as for the counting function $R_D^\loc (N)$ in section \ref{local}. We have
\begin{equation*}
\RBR (N)= \sum_{\bfl\in \N^r} \sum_{(a,b)\in \calH(\bfl)} \!\!\RBR \left(N;\prod_{i=1}^r p_i^{l_i},a,b\right)+O(N).
\end{equation*}
The term $O(N)$ here comes from all the tuples $(A,B)$, for which one of the $G_j(A,B)=0$. We next truncate the sum at a positive integer $L$. We use the vector notation $1\leq \bfl\leq L$ to express that $1\leq l_i\leq L$ for all $1\leq i\leq r$. We rewrite the expression for $\RBR (N)$ as
\begin{equation*}
\RBR (N)= \sum_{1\leq \bfl \leq L} \sum_{(a,b)\in \calH(\bfl)} \!\!\RBR\!\! \left(N;\prod_{i=1}^r p_i^{l_i},a,b\right)+O(N)+E_5,
\end{equation*}
with
\begin{equation*}
E_5 \ll \sum_{i=1}^r \sum_{j=1}^m \sharp\{|A|,|B|\leq N: p_i^L| G_j(A,B)\}.
\end{equation*}
Here the polynomials $G_j(A,B)$ are defined as in Lemma \ref{lem1}. As in the proof of Lemma \ref{lem1b}, there is a real constant $\tet_1>0$ such that
\begin{equation*}
\sharp\{|A|,|B|\leq N: p_i^L |G_j(A,B)\}\ll N^2p_i^{-\tet_1 L}+Np_i^{2L}.
\end{equation*}
Let $A_0>0$ be a real parameter to be chosen later. We let $L$ be the largest integer such that  $\smash{p_r^L\leq (\log N)^{A_0/r}}$. In particular, we have $\smash{L\leq \frac{A_0 \log \log N}{r\log p_r}}$. We hence may apply Theorem \ref{thm5} to evaluate $\smash{\RBR (N;\prod_{i=1}^r p_i^{l_i},a,b)}$ and obtain
\begin{equation*}
\RBR (N)= 4N^2  \sum_{1\leq \bfl\leq L} \sum_{(a,b)\in \calH(\bfl)}\sum_{k=0}^{2P} \frac{c_k(T'')}{T\phi(T'')T' \left(\log \frac{2N}{T'}\right)^{1/4+k/2}} +O(N)+ E_5+E_6,
\end{equation*}
with an error term $E_6$ bounded by
\begin{equation*}
E_6\ll_{D,P,A_0} (\log \log N)^{r} \frac{N^2}{(\log N)^{P+3/4}}.
\end{equation*}
We next develop the expression 
\begin{equation*}
\frac{1}{(\log (2N)- \log T')^{1/4+k/2}}
\end{equation*}
into a series of powers of $\log 2N$ and hence may rewrite this as
\begin{equation*}
\RBR (N)= 4N^2  \sum_{k=0}^{2P}\sum_{1\leq \bfl\leq L} \sum_{(a,b)\in \calH(\bfl)} \frac{c_k'(T'')}{T\phi(T'')T' \left(\log 2N\right)^{1/4+k/2}} + E_5+ E_6,
\end{equation*}
with coefficients $c_k'(T'')\ll_{D,P,A_0} (\log \log N)^{2P}$ and $c_0'(T'')= c_0(T'')$ for all $T''$. At this point, we also note that $T'$ and $T''$ in general depend on $(a,b)$ by Theorem \ref{thm5}.\par
We claim that the series
\begin{equation*}
\sum_{\bfl\in \N^r} \sum_{(a,b)\in \calH(\bfl)} \frac{c_k'(T'')}{T\phi(T'')T'}
\end{equation*}
is absolutely convergent. Indeed, by Lemma \ref{lem1b}, for any $1\leq i\leq r$, we have the estimate
\begin{equation*}
\begin{split}
\sum_{\substack{\bfl\in \N^r\\ l_i>L}} \sum_{(a,b)\in \calH(\bfl)} \frac{c_k'(T'')}{T\phi(T'')T'} & \ll_{D,P,A_0} \sum_{\substack{\bfl\in \N^r\\ l_i>L}}  \frac{(\log \log N)^{2P}}{\prod_{i=1}^r p_i^{2l_i}} \sharp \calH(\bfl)  \\ &\ll_{D,P,A_0} (\log \log N)^{2P} \sum_{\substack{\bfl\in \N^r\\ l_i>L}}\prod_{i=1}^r p_i^{-\tet_0l_i}\\  &\ll_{D,P,A_0} (\log \log N)^{2P} p_i^{-\tet_0 L}.
\end{split}
\end{equation*}
If we choose $A_0$ sufficiently large, we hence obtain
\begin{equation*}
\RBR (N)= \frac{4N^2}{(\log 2N)^{1/4}} \sum_{k=0}^{2P} \frac{C_k}{(\log 2N)^{k/2}} +O_{D,P} \left(\frac{N^2}{(\log N)^{3/4+P}}\right),
\end{equation*}
with constants $C_k$ of the form
\begin{equation}\label{defCk}
C_k= \sum_{\bfl\in \N^r} \sum_{(a,b)\in \calH(\bfl)} \frac{c_k'(T'')}{T\phi(T'')T'}.
\end{equation}
Moreover, for $k=0$, we specifically obtain
\begin{equation}\label{defC0}
C_0= \sum_{\bfl\in \N^r} \sum_{(a,b)\in \calH(\bfl)} \frac{G(1,T'')}{\Gam\!\left(\frac{3}{4}\right) T\phi(T'')T'},
\end{equation}
where $G(1)=G(1,T'')$ is defined as in Lemma \ref{lem5}. The proof of Theorem \ref{theo1} is now completed by Lemma \ref{Cpos}.\vspace{0.2cm}\qedsymbol

Before we prove that the leading constant $C_0$ is indeed positive, we prepare with two lemmata. The first of them is a modified version of Lemma 6.7 in \cite{JS}. 

\begin{lem}\label{lem6.7}
Let $p>9$ be a prime and $\F_{\!p}$ be the finite field with $p$ elements. Then there are elements $a_0$ and $a_1\in \F_{\!p}$ with the following properties. Both $a_0$ and $a_1$ are squares different from $0,-1$, with $a_i^2+a_i+1\neq 0$, and such that $a_0+1$ is a square, and $a_1+1$ is a non-square.
\end{lem}

\begin{proof}
We only consider the case of $a_0$, since the arguments for $a_1$ are identical. To establish the claim in the lemma, it is sufficient to find a (non-trivial) point on the conic $u^2+w^2=v^2$ over $\F_{\! p}$, with $w\neq 0$, $\smash{(\frac{u}{w})^2\neq 0,-1}$ and $\smash{(\frac{u}{w})^4+(\frac{u}{w})^2+1\neq 0}$. In the projective plane, the conic $u^2+w^2=v^2$ has exactly $p+1$ points. There are at most two points with $w=0$, at most four points with $u=0$ or $\smash{(\frac{u}{w})^2=-1}$, and at most four points satisfying $w\neq 0$ and $\smash{(\frac{u}{w})^4+(\frac{u}{w})^2+ 1= 0}$. Hence there is a point with the desired properties as soon as $p+1>10$. 
\end{proof}


\begin{lem}\label{lem6.8}
Assume that $3|D$, $A\equiv -D$ modulo $9$ and $B\equiv 0 $ modulo $9$. Then $S^{(D;A,B)}(\Q_3)\neq \emptyset$, and the Brauer class $\alp$ evaluates constantly to zero on $S^{(D;A,B)}(\Q_3)$.
\end{lem}

\begin{proof}
The existence of some point in $S^{(D;A,B)}(\Q_3)$ is clear since $(1:1:1:0:0)$ is a smooth point on the reduction of $S^{(D;A,B)}$ to $\F_{\!3}$. Hence, we need to show that $\alp$ evaluates constantly. For this, let $\bft\in S^{(D;A,B)}(\Q_3)$, and assume the $t_i$ normalised s.t.~$t_i\in \Z_3$ and one of them is a unit. If $3|t_1$, then the first equation in (\ref{eqn0}) shows that $3|t_2$ and hence by the second equation yields $3|t_0$. Since $D$ is assumed to be squarefree, this leads to all of the $t_i$ being divisible by $3$, a contradiction. Hence, we may assume without loss of generality that $t_1=1$. Now, the first equation in (\ref{eqn0}) shows that $t_0$ is a norm, and hence $t_0\equiv 1 \mmod 3$ or $t_0\equiv -D\mmod 9$ or $t_0\equiv 0 \mmod 9$. In the first case, one has $\smash{\frac{t_0+At_1}{t_1}\equiv 1 \mmod 3}$, which is a norm. In the second case, one has $\smash{\frac{t_0+Bt_1}{t_1}\equiv -D \mmod 9}$, and in the third case $\smash{\frac{t_0+At_1}{t_1}\equiv - D \mmod 9}$, which are both norms, as well. Hence $\alp$ evaluates constantly to 0 on $S^{(D;A,B)}(\Q_3)$.
\end{proof}

We can now show that the leading constant $C_0$ is indeed positive.

\begin{lem}\label{Cpos}
One has $C_0>0$.
\end{lem}

\begin{proof}
Recall the definition of $C_0$ in equation (\ref{defC0}). By Lemma \ref{lem5}, we see that each of the $G(1,T'')>0$, such that the problem reduces to showing that there is some $\bfl\in \N^r$ such that $\calH(\bfl)\neq \emptyset$. For this, we construct a tuple of integers $(A,B)$ satisfying the following properties:\vspace{0.1cm}\\
i) If $p_1=3$, then $A\equiv -D$ modulo $9$ and $B\equiv 0 $ modulo $9$.\vspace{0.1cm}\\
ii) For $p_i>3$, the residue class $\bar A= (A\mmod p_i)$ is a square, different from $0,-1$, and such that $\bar A^2+\bar A +1\neq 0$. Furthermore $\smash{B\equiv -\frac{A}{A+1}\mmod p_i}$.\vspace{0.1cm}\\
iii) If there is an even number of non-squares among $(A\mmod p_i)+1$ for primes $p_i>3$ and $i<r$, then $(A\mmod p_r)+1$ is a non-square, and if there is an odd number of non-squares among $(A\mmod p_i)+1$ for primes $p_i>3$ and $i<r$, then $(A\mmod p_r)+1$ is a square.\vspace{0.1cm}\\
iv) All of the polynomials $G_j(A,B)$ as defined in Lemma \ref{lem1} are non-zero.\vspace{0.1cm}\\
By Lemma \ref{lem6.7}, such a choice for $(A,B)$ is possible. This is clear for $D\neq 3\cdot 5\cdot 7$. For $D=3\cdot 5\cdot 7$ we note that condition ii) forces $(A\mmod 5)=1$ and hence $(A\mmod 5)+1$ is a non-square. Then, over the field $\F_{\! 7}$, there is an element $a_0\neq 0,-1$ with $a_0^2+a_0+1\neq 0$, and such that $a_0+1$ is a square, take e.g.~$a_0=1$.\par
If $3|D$ then, by Lemma \ref{lem6.8}, condition~i) implies that $S^{(D;A,B)}(\Q_3)\neq \emptyset$. Furthermore, the Brauer class $\alp$ evaluates constantly to zero on $S^{(D;A,B)}(\Q_3)$. Since none of the $G_j(A,B)$ vanish, this implies, together with Proposition 5.1 in \cite{JS}, that there is some $\bfl\in \N^r$ such that the reduction of $(A,B)$ modulo $\smash{\prod_{i=1}^r p_i^{l_i}}$ is contained in $\calH(\bfl)$. Hence we have $\calH (\bfl)\neq \emptyset$, which completes the proof of the lemma.
\end{proof}


\setlength\parindent{0mm}
\end{document}